\documentclass[11pt]{amsart}

\usepackage[utf8]{inputenc}
\usepackage[T1]{fontenc}
\usepackage[UKenglish]{babel}
\usepackage{times}

\usepackage{amsfonts,amssymb,stmaryrd}
\usepackage{mathrsfs}

\usepackage[a4paper,top=2cm,bottom=2cm,left=2cm,right=2cm,marginparwidth=2cm]{geometry}

\usepackage{amsmath,amsthm,upgreek,mathtools}
\usepackage{tikz-cd}
\usepackage{changepage}
\usepackage{cancel}
\usepackage{enumerate}
\usepackage[all]{xy}

\usepackage{verbatim}
\usepackage[normalem]{ulem}

\usepackage{graphicx}
 
\usepackage[colorinlistoftodos,textsize=tiny,textwidth=2.5cm]{todonotes}

\usepackage{appendix}
\usepackage{multicol}

\usepackage[colorlinks=true, allcolors=blue,backref=page]{hyperref}
\usepackage[initials,alphabetic]{amsrefs}

\newcommand{\beq}{\begin{equation}}
\newcommand{\bea}{\begin{eqnarray}}
\newcommand{\eea}{\end{eqnarray}} 
\newcommand{\eeq}{\end{equation}}
\newcommand{\ddb}{{\partial\delbar}}

\newcommand{\rd}{{\rm d}}

\newcommand{\be}{{\bf e}}

\newcommand{\cF}{\mathcal{F}}

\newcommand{\cO}{\mathcal{O}}
\newcommand{\cP}{\mathcal{P}}

\newcommand{\cR}{\mathcal{R}}

\newcommand{\sE}{\mathscr{E}}
\newcommand{\sF}{\mathscr{F}}

\newcommand{\sH}{\mathscr{H}}

\newcommand{\sT}{\mathscr{T}}

\newcommand{\R}{\mathbb{R}}
\newcommand{\C}{\mathbb{C}}

\newcommand{\SU}{{\rm SU}}

\newcommand{\End}{{\mathrm{End}}}

\renewcommand{\epsilon}{\varepsilon}

\newcommand{\delbar}{\bar \del}
\newcommand{\del}{\partial}

\newcommand{\id}{\mathrm{id}}
\newcommand{\im}{\mathop{\mathrm{im}}}

\newcommand{\supp}{\mathop{\mathrm{supp}}}
\newcommand{\tr}{\mathop{\mathrm{tr}}\nolimits}

\newcommand{\qandq}{\quad\text{and}\quad}
\newcommand{\qforq}{\quad\text{for}\quad}

\def\<{\mathopen{}\left<}
\def\>{\right>\mathclose{}}
\def\({\mathopen{}\left(}
\def\){\right)\mathclose{}}

\usepackage{multicol, color}

\definecolor{gold}{rgb}{0.85,.66,0}
\definecolor{cherry}{rgb}{0.9,.1,.2}
\definecolor{burgundy}{rgb}{0.8,.2,.2}
\definecolor{orangered}{rgb}{0.85,.3,0}
\definecolor{orange}{rgb}{0.85,.4,0}
\definecolor{olive}{rgb}{.45,.4,0}
\definecolor{lime}{rgb}{.6,.9,0}
\definecolor{green}{rgb}{.2,.7,0}
\definecolor{grey}{rgb}{.4,.4,.2}
\definecolor{brown}{rgb}{.4,.3,.1}

\newtheorem{theorem}{Theorem}[section]
\newtheorem{prop}[theorem]{Proposition}
\newtheorem{corollary}[theorem]{Corollary}
\newtheorem{lemma}[theorem]{Lemma}
\theoremstyle{remark}
\newtheorem{remark}[theorem]{Remark}
\theoremstyle{definition}
\newtheorem{definition}[theorem]{Definition}

\newtheorem{notation}[theorem]{Notation}
\numberwithin{equation}{section}

\title{Local descriptions of the heterotic SU(3) moduli space}

\author{Hannah de Lázari, Jason D. Lotay, Henrique Sá Earp, Eirik Eik Svanes}
\date{\today}

\begin{document}

\begin{abstract}
    The heterotic $\SU(3)$ system, also known as the Hull--Strominger system, arises from compactifications of heterotic string theory to six dimensions. This paper investigates the local structure of the moduli space of solutions to this system on a compact 6-manifold $X$, using a vector bundle $Q=(T^{1,0}X)^* \oplus {\End}(E) \oplus T^{1,0}X$, where $E\to X$ is the classical gauge bundle arising in the system. We establish that the moduli space has an expected dimension of zero.  We achieve this by studying the deformation complex associated to a differential operator $\bar{D}$, which emulates a holomorphic structure on $Q$, and demonstrating an isomorphism between the two cohomology groups which govern the infinitesimal deformations and obstructions in the deformation theory for the system.  We also provide a Dolbeault-type theorem linking these cohomology groups to \v{C}ech cohomology, a result which might be of independent interest, as well as potentially valuable for future research.
\end{abstract}

\maketitle

\begin{adjustwidth}{0.85cm}{1.20cm}
    \tableofcontents
\end{adjustwidth}

%\newpage
%\listoftodos

\newpage
\section{Introduction}

The heterotic $\SU(3)$ (or Hull--Strominger) system arises as an approximation to the equations of motion associated to heterotic String Theory in 10 dimensions, when compactified to 6 dimensions (cf.~\cites{Strominger:1986uh,Hull1986,Ossa2014a, GarciaFernandez2016}). As such, the system is therefore of importance in high energy theoretical physics.  Mathematically, it is an intricate coupled system for an ambient $\SU(3)$-structure on a real $6$-manifold $X$ and instanton (Hermitian Yang--Mills) connections on bundles over $X$.  One subtlety is that whilst the equations on the ambient structure and connections are \emph{first order} conditions, the coupling happens at \emph{second order} via the \emph{anomaly condition},  also called the \emph{heterotic Bianchi identity}.  Specifically, the anomaly condition intertwines derivatives of the $\SU(3)$-torsion  with the curvatures of the instanton connections and involves a real constant $\alpha'$, which we assume to be  non-zero (since the case where $\alpha'=0$ is a degenerate situation both mathematically and physically).  This is fascinating mathematically, since the question of existence for the specific classes of $\SU(3)$-structures and instanton connections are well-understood, yet we have very little knowledge about the corresponding question about the heterotic $\SU(3)$ system.  

Another important mathematical motivation for the heterotic $\SU(3)$ system arises from \emph{Reid's fantasy}: the hope that one might be able to connect any two compact complex 3-folds with trivial canonical bundle  through a sequence of procedures known as \emph{conifold transitions}, where the manifold degenerates to form singularities and then is resolved with a topological change.  The key point is that, though conifold transitions preserve the triviality of the canonical bundle, the K\"ahler condition may not be preserved, and in fact cannot be preserved in some cases for purely topological reasons.  It has been advocated that the heterotic $\SU(3)$ system provides a natural candidate for a canonical geometry on such compact complex 3-folds (often called \emph{non-K\"ahler Calabi--Yau}), which may give an approach to understanding Reid's fantasy: see e.g. \cites{CollinsPicardYau2021,FriedmanPicardSuan2024} for recent progress.

Our current understanding of the space of heterotic $\SU(3)$ solutions on compact (or even non-compact) manifolds is very limited. A fundamental question is whether the space is non-empty: see e.g~\cites{GarciaFernandez2020b, Fino2021, CollinsPicardYau2022,   GarciaFernandez2023} for some recent progress.  If it is non-empty, then the next natural question is about the local structure of the moduli space of solutions, which is the focus of this paper.  

\subsection{Results} Our main result is the following. 
\begin{theorem}
\label{thm:main} 
    Let $X$ be a compact 6-manifold endowed with a solution of the heterotic $\SU(3)$ system which we denote by $s_0$.  There exist vector spaces $\mathcal{I}$ and $\mathcal{O}$ of the same finite dimension, an open set $U\ni 0$ in $\mathcal{I}$ and a smooth map $\pi:U\to \mathcal{O}$ with $\pi(0)=0$ such that any point in the moduli space of solutions to the heterotic $\SU(3)$ system near $s_0$ corresponds to an element in $\pi^{-1}(0)$. 
\end{theorem}
Since generic points in $\mathcal{O}$ will be regular values of $\pi$ by Sard's theorem, we may infer that $\pi^{-1}(0)$, and hence the moduli space of solutions to the heterotic $\SU(3)$ system, has expected dimension zero in this sense. (Another approach to the expected dimension of the moduli space is through derived geometry, cf.~\cite{Tellez-Dominguez:2023wwr}). 
\begin{remark}
    The notation in Theorem \ref{thm:main} is suggestive: $\mathcal{I}$ is for infinitesimal deformation space and $\mathcal{O}$ is for obstruction space. The map $\pi$ may be referred to as the obstruction or Kuranishi map.
\end{remark}

Theorem \ref{thm:main} is somewhat surprising, given that there are certainly known deformation families of solutions to the heterotic $\SU(3)$ system arising from Calabi--Yau 3-folds, but one should expect these to be non-generic (and so, in particular, $0$ would not be a regular value of the map $\pi$ in these cases).    It also demonstrates the importance of the anomaly condition, since the conditions on the ambient $\SU(3)$-structure are undetermined and so lead to infinite-dimensional families, which will also admit instanton connections.  

Theorem \ref{thm:main} also has clear implications for the approach to Reid's fantasy for non-K\"ahler Calabi--Yau 3-folds using the heterotic $\SU(3)$ system.  It suggests that one should not expect to be able connect any two solutions to the system along a path, even supposing there is no need for conifold transitions, as the existence of such a path would be non-generic.  On the other hand, Theorem \ref{thm:main} also raises the possibility of a solution to this issue: namely, to allow the real constant $\alpha'$ appearing in the anomaly condition to vary along the path.  Since the case $\alpha'=0$ represents a degenerate situation not covered by our analysis, this observation further motivates the investigation of the behaviour of solutions of the heterotic $\SU(3)$ system as $\alpha'\to 0$, which is an important and essentially unexplored aspect of the field.

\begin{remark} 
    In view of  Theorem \ref{thm:main}, one may speculate as to whether one may ``count'' the number of solutions to the heterotic $\SU(3)$ system on a compact 6-manifold $X$ and whether this carries any interesting information about $X$ or the chosen bundles over it.  Defining such a count motivates the study of the compactness theory for solutions of the heterotic $\SU(3)$ system.  If a well-defined number can be assigned to the solution space, the next intriguing question would be whether it is invariant under suitable perturbations. 
\end{remark}
\begin{remark} 
    Theorem \ref{thm:main} is also interesting from a physical point of view. Indeed, it appears very difficult in string theory to find (perturbative) string compactifications where all the moduli are ``stabilised''. This translates mathematically to having a zero-dimensional moduli space, and is known as the string theory moduli problem. One might hope that the moduli are stabilised by some unknown higher order perhaps non-perturbative (quantum) effect, and the fact that the moduli space has expected dimension zero  lends credence to this idea. See \cite{Chisamanga:2024xbm} for a more in-depth physics discussion of this point. 
\end{remark}

\begin{remark}
The map $\pi$ in Theorem \ref{thm:main} is not yet understood. In particular,  the moduli space of solutions to the heterotic $\SU(3)$ system may yet be positive-dimensional.  For heterotic $\SU(3)$ solutions coming from Calabi--Yau 3-folds $\pi$ will be the zero map,  which will then always yield a positive-dimensional moduli space as a consequence of the Bogomolov--Tian--Todorov theorem but, as remarked earlier, one might expect these to be non-generic solutions.  However, there are other situations where the moduli space can be positive-dimensional,  since $\pi$ vanishes on a positive-dimensional submanifold of $\mathcal{I}$ \cite{GarciaFernandez2022}*{Theorem 5.8}.  Therefore, an important next step in the study of the heterotic $\SU(3)$ system would be to examine the map $\pi$ further.
\end{remark}

To provide some more detail on our work, it was shown in \cite{McOrist2022} that the moduli space of solutions is governed by a certain first-order differential operator $\bar{D}$ acting on sections of a specific vector bundle $Q\to X^6$, which is related to the tangent bundle and the gauge bundle on which the instanton connections are defined. In particular, the infinitesimal deformations were shown to correspond to elements of a cohomology group $H^{0,1}_{\bar{D}}(Q)$, i.e.~the kernel of $\bar{D}$ acting on $Q$-valued $(0,1)$-forms, modulo the image of $\bar{D}$ acting on sections of $Q$.  Moreover, though the full Maurer--Cartan equation for deformations is yet to be worked out, it is shown that the potential obstructions lie in $H^{0,2}_{\bar{D}}(Q)$, which is defined in a similar manner. We prove the following when $X$ is compact, which can be viewed as an analogue of Serre duality in this context.
\begin{theorem}\label{thm:index}
    The  cohomology groups $H^{0,1}_{\bar{D}}(Q)$ and $H^{0,2}_{\bar{D}}(Q)$ are isomorphic. 
\end{theorem}
 \noindent Theorem \ref{thm:main} then follows immediately from Theorem \ref{thm:index}, by the work in \cite{McOrist2022} and the discussion in Appendix \ref{app:Moduli}, with $\mathcal{I}=H^{0,1}_{\bar{D}}(Q)$ and $\mathcal{O}=H^{0,2}_{\bar{D}}(Q)$.  We further prove that, in cases of physical interest, the cohomology groups in Theorem \ref{thm:index} are isomorphic to the kernel of $\bar{D}+\bar{D}^*$ acting on appropriate $Q$-valued forms, by showing that $\bar{D}+\bar{D}^*$ is (overdetermined) elliptic.
 
Since the cohomology groups in Theorem \ref{thm:index} are defined using a non-standard differential operator, it is useful to have an alternative description to potentially facilitate computation.  To this end, we define certain \v{C}ech cohomology groups $\check{H}^p(Q)$, which can be interpreted as sheaf cohomology groups, and show the following Dolbeault-type theorem.

\begin{theorem}\label{thm:Cech}
    The cohomology groups $H^{0,1}_{\bar{D}}(Q)$ and $\check{H}^1(Q)$ are isomorphic.
\end{theorem}

\noindent We will prove Theorem \ref{thm:Cech} by constructing an explicit isomorphism between $H^{0,1}_{\bar{D}}(Q)$ and $\check{H}^1(Q)$. Theorem \ref{thm:Cech} then should enable us to use topological and algebraic methods to ascertain whether $H^{0,1}_{\bar{D}}(Q)$ and $H^{0,2}_{\bar{D}}(Q)$ both vanish or not (by Theorem \ref{thm:index}).
 
\subsection{Double extensions and self-duality}\label{ss:double.extension}  

We now make some further observations about $Q$ and $\bar{D}$ introduced above, which motivated our study.  The vector bundle is explicitly given as 
$$
Q=(T^{1,0}X)^* \oplus {\End}(E) \oplus T^{1,0}X,
$$ 
and the differential operator $\bar{D}$ allows us to realise $Q$  as a ``double extension'':
\begin{equation}
\label{eq:Atiyah}
\xymatrix @R=0.8pc  {
    0 \ar[r] & %\Omega^{0,p}(
    \End(E)%) 
    \ar[r] & %\Omega^{0,p}(
    Q_1%) 
    \ar[r] & %\Omega^{0,p}(
    T^{1,0}X%) 
    \ar[r] & 0; \\
    0 \ar[r] & %\Omega^{0,p}(
    (T^{1,0}X)^*%) 
    \ar[r] & %\Omega^{0,p}(
    Q%) 
    \ar[r] & %\Omega^{0,p}(
    Q_1%) 
    \ar[r] & 0.
}
\end{equation}
Here, the first extension defining $Q_1$ is often referred to as the \emph{Atiyah extension}, which allows us to control simultaneous deformations of a complex manifold equipped with a holomorphic bundle \cite{Atiyah1957}. Alternatively, we can also construct $Q$ by the following double extension:
\begin{equation}\label{eq:Atiyah.dual}
\xymatrix @R=0.8pc {
    0 \ar[r] & (T^{1,0}X)^* \ar[r] & Q_1^* \ar[r] & \End(E) \ar[r] & 0; \\
    0 \ar[r] & Q_1^* \ar[r] & Q \ar[r] & T^{1,0}X \ar[r] & 0.
}
\end{equation}
Note that the extension defining $Q_1^*$ is precisely the dual of the Atiyah extension. In this sense, the bundle $Q$ is ``self-dual''.

As we see from Theorem \ref{thm:index}, Theorem \ref{thm:main} follows from an isomorphism $H^{0,1}_{\bar{D}}(Q)\simeq H^{0,2}_{\bar{D}}(Q)$.  
In analogous situations,  when the ambient 6-manifold $X$ is a Calabi--Yau 3-fold, given a self-dual bundle $\tilde{Q}\to X$ one can show that $H^{0,2}_{\bar{D}}(\tilde{Q})\simeq H^{0,1}_{\bar{D}}(\tilde{Q})$ using Serre duality. However, in our setting $X$ is not necessarily Calabi--Yau and, more importantly, $\bar D$ is \emph{not} a connection on the bundle $Q$ in the usual sense, so self-duality is not \emph{a priori} a well-defined notion for $Q$. In its stead, our approach was to establish the isomorphism directly,  by mediation of  auxiliary \v{C}ech cohomology groups.

\subsection{The heterotic SU(3) or Hull--Strominger system} 
The main object of consideration in this paper is the heterotic $\SU(3)$ (or Hull--Strominger) system \cites{Hull1986, Strominger:1986uh}, which derives from supersymmetric solutions to heterotic supergravity on complex $3$-folds $X$.  Following \cite{Ossa2014a}, the manifold $X$ is endowed with the following data:
\begin{itemize} 
    \item an $\SU(3)$-structure defined by a Hermitian $(1,1)$-form $\omega$ and nowhere vanishing $(3,0)$-form $\Omega$;
    \item the Chern connection $\nabla$ on $TX$, with curvature $R$;
    \item a holomorphic vector bundle $E\to X$, with a Hermitian metric on its fibres,  together with the (unitary) Chern connection $A$ on $E$, with curvature $F$;
    \item a non-zero real constant $\alpha'$. 
\end{itemize} 
Physics requires $E$ to have a structure group contained in $E_8\times E_8$ or $\textrm{Spin}(32)/\mathbb{Z}_2$, and $\alpha'$ to be positive and ``small''.  The size of $\alpha'$ is reflected in the fact that the system of equations we now describe is derived by ignoring all terms of order $O(\alpha')^2$.

These structures on $X$ are required to satisfy the following differential constraints, often referred to as the F-term constraints:
\begin{align}
    \textrm{d}\Omega&=0;\label{eq:F1}\\
    2i\del\delbar\omega&=\alpha'\big(\tr(F\wedge F)-\tr(R\wedge R)\big).\label{eq:F2}
\end{align}
    Note that equation \eqref{eq:F1} forces $c_1(X)=0$.  

  Equation \eqref{eq:F2} is referred to as the \emph{anomaly cancellation} condition, or the \emph{heterotic Bianchi identity}.  In physics, one may alternatively find the anomaly condition expresssed at the level of primitives for the terms arising in \eqref{eq:F2}.  Explicitly, if $\text{CS}(A)$ denotes the Chern--Simons 3-form associated to $A$, given by 
\begin{equation*}
\text{CS}(A)=\tr(A\wedge\textrm{d} A+\frac{2}{3}A\wedge A\wedge A)\:,
\end{equation*}
which satisfies
\[
\textrm{d}\text{CS}(A)=\tr(F\wedge F)\:,
\]
and similarly define $\text{CS}(\nabla)$, then we may impose the condition
\begin{equation}\label{eq:anomaly.def}
2\textrm{d}^c\omega=\alpha'(\text{CS}(A)-\text{CS}(\nabla))+\textrm{d} B\:,
\end{equation}
where $B$ is a multi-valued 2-form which can be used to account for the fact that the Chern--Simons forms are multi-valued, whereas $\textrm{d}^c\omega$ is single-valued. We may then note that \eqref{eq:anomaly.def} implies \eqref{eq:F2}.

In addition, the geometric structures on $X$ are  required to satisfy the following D-term constraints:
\begin{align}
    F\wedge \omega\wedge\omega&=0;\label{eq:D1}\\
    \textrm{d}(\vert\Omega\vert_{\omega}\omega\wedge\omega)&=0.\label{eq:D2}
\end{align}
Given the initial assumptions on $A$, \eqref{eq:D1}  is the (degree zero) Hermitian Yang-Mills instanton condition on $A$, while \eqref{eq:D2} is known as a \emph{conformally balanced} condition for the $\SU(3)$-structure. The D-terms will be less relevant for the analysis of this paper, but we include them here for completeness.

\begin{remark}
    The equations \eqref{eq:F1}--\eqref{eq:D2} form the system Strominger defined in the seminal paper \cite{Strominger:1986uh}. Hull \cite{Hull1986} used the more physically accurate Hull connection on $TX$, in place of the Chern connection, to define the curvature term $\tr(R\wedge R)$ in \eqref{eq:F2}. However,  this choice leads to a system of equations which is not mathematically closed: e.g.~the curvature of the Hull connection is only type $(1,1)$ modulo higher order $\alpha'$ corrections.
\end{remark}

The equations \eqref{eq:F1}--\eqref{eq:D2} for the given data we have described on $X$ are commonly referred to as the \mbox{\emph{(Hull--)Strominger system}}  in the mathematics literature, e.g.~in the study by Fu--Yau of solutions to the system \cite{Fu2008}, or in the analysis of the associated anomaly flow in \cite{Phong:2016ggz}.      We shall therefore refer to \eqref{eq:F1}--\eqref{eq:D2} for the data $\omega$, $\Omega$, $(TX,\nabla)$, $(E,A)$ and $\alpha'$ as the \emph{heterotic $\SU(3)$ system}.

\begin{remark}
\label{rem:HYM1} 
    If one wants the solution of \eqref{eq:F1}-\eqref{eq:D2} to lead to a solution of the heterotic equations of motion, it is known that the connection on $TX$ whose curvature $R$ appears in \eqref{eq:F2} should also be an instanton \cites{Ivanov2010, Martelli2011}, but this is generally not the case for the Chern connection $\nabla$. One common fix to this issue is to replace $\nabla$ by an additional instanton connection $\tilde\nabla$ on $TX$, see e.g.~\cites{Anderson2014, Ossa2014, GarciaFernandez2017}. Introducing such a new degree of freedom is however not warranted from a physics perspective, so we do not take this approach here.  Yet, we shall explain how one can modify our analysis to accommodate this choice, in Remark \ref{rem:HYM2}.
\end{remark}

As previously stated, we base our analysis of the moduli space of heterotic $\SU(3)$ solutions on \cite{McOrist2022}.  This moduli space was also studied in the earlier works of \cites{Anderson2014, Ossa2014a, GarciaFernandez2017} but, as mentioned in Remark \ref{rem:HYM1}, these sources employ a mathematical framework that promotes a specific tangent bundle connection to an extra degree of freedom for the heterotic $\SU(3)$ system. This makes the moduli problem mathematically easier, but leaves extra moduli deformation parameters which are ``non-physical'' and thus should be disregarded.

We should advise the reader that the analysis of \cite{McOrist2022} was performed within the framework of physics; in particular, terms of order $O(\alpha'^2)$ were ignored, just as in the derivation of the heterotic $\SU(3)$ system itself.  Further details, as well as progress towards removing this assumption, are given in Appendix \ref{app:Moduli}. In any case, we believe that a full treatment of the moduli problem following the formalism of \cite{McOrist2022} will reveal a differential complex and a cohomology which, if not exactly the same, will have the same features as $H^{0,p}_{\bar{D}}(Q)$. Analysing the mathematical properties of $\bar D$ is therefore a useful and worthwhile pursuit, which will illuminate the moduli problem for the heterotic $\SU(3)$ system. 

\begin{remark}
    All our results  require $\alpha'\neq 0$, which is the situation of primary significance in mathematics and physics.  However, one is often interested in considering the limit as $\alpha'$ approaches zero, so it is important to be aware that certain parts of our work do not hold at the limit $\alpha'=0$.  This may be significant for the intriguing problem of understanding the limiting behaviour of solutions to the heterotic $\SU(3)$ system as $\alpha'\to 0$.
\end{remark}

\noindent\textbf{Acknowledgements:} The authors would like to thank Luis Álvarez-Cónsul, Beatrice Chisamanga, Mario Garcia-Fernandez, Raul Gonzalez Molina, Jock McOrist, Javier José Murgas Ibarra, Sébastien Picard and Markus Upmeier for valuable discussions. We particularly thank Javier José Murgas Ibarra for valuable insights into defining the notion of \v{C}ech cohomology considered in this paper. ES would like to thank the mathematical research institute MATRIX in Australia, where part of this research was performed, for a lively and rewarding environment. 

\bigskip
%%%%%%%%
\section{The deformation operator}

In this section we study the first order differential operator $\bar D$ on $Q$-valued forms, which will control the deformation theory of solutions to the heterotic $\SU(3)$ system.  After recalling its definition, we will show that $\bar D^2=0$, which will in fact be equivalent to the anomaly cancellation condition in the heterotic $\SU(3)$ system. This coboundary property enables us to define the cohomology groups $H^{0,p}_{\bar{D}}(Q)$.  We will then compute the formal adjoint $\bar D^*$ of $\bar D$, and finally show that $\bar D+\bar D^*$ defines an (overdetermined) elliptic operator.  This leads us to conclude that $H^{0,p}_{\bar{D}}(Q)$ is isomorphic to the kernel of $\bar{D}+\bar{D}^*$ acting on $Q$-valued $(0,p)$-forms.

\subsection{Preliminaries}
 We begin by recalling the set up for the deformation theory for the heterotic $\SU(3)$ system as in \cite{Ossa2014a}.  Throughout the article we shall use the following notation.

 \begin{notation}\label{notation1}
     Let $X$ be a complex 3-manifold with local holomorphic coordinates denoted by $(z_1,z_2,z_3)$.  
  \begin{itemize}   
   \item   Let $\omega$ be a Hermitian metric on $X$, viewed as a $(1,1)$-form, which we      write locally as $i g_{j \bar{k}} dz^j \wedge d \bar{z}^k$ where, here and throughout, we use summation convention.   Note that $g_{jk}$ are the local coefficients of the metric $g$ associated to $\omega$.
   \item Let $\nabla$ denote the Chern connection of the metric $g$ on $X$, which extends naturally to all differential forms on $X$, and let $R$ be the curvature of $\nabla$, whose local coefficients satisfy
\begin{equation}
\label{R-conv}
    [\nabla_j,\nabla_{\bar{k}}]V^l 
    =-R_{\bar{k}j}{}^l{}_mV^m,
   \end{equation}
   for a local $(1,0)$-vector field with coefficients $V^l$.
   \item Let $T=i\partial\omega$, which is often called the torsion, be given locally by $T_{\bar{l}jk}d\bar{z}^{\bar{l}}\wedge dz^j\wedge dz^k$.
   This is the $(2,1)$-part of the heterotic NS flux $H=T+\bar T$. 
   \item Let $\Omega$ be a nowhere vanishing $(3,0)$-form on $X$.
   \item  
     Let $E \rightarrow X$ be a holomorphic vector bundle with fibrewise Hermitian metric $h$, whose local coefficients are $h_{jk}$, with respect to a local frame $\{e_j\}$ on $E$ and its dual coframe $\{e^j\}$. We assume the structure group of $E$ is such that endomorphisms of $E$ are trace-free.
 \item    
Let $A$ be the unique unitary integrable (Chern) connection on $E$, with curvature  $F\in\Omega^{1,1}(\End E)$, which  is written locally as $F_{j\bar{k}}dz^j\wedge d\bar{z}^k$ for endomorphisms $F_{j\bar{k}}$.
    \item Let $\delbar$ denote the Dolbeault operator acting on tensors on $X$ and let $\delbar_E$ denote the Dolbeault operator arising from the holomorphic structure on $E$ which acts on tensor-valued sections of $\mathrm{End}(E)$.
\item Let $\alpha'\in\R\setminus\{0\}$.  For physical relevance, we shall be primarily interested in the case when $\alpha'>0$ and small relative to the other quantities defined on $X$.  
\end{itemize}

We assume that this data satisfies the heterotic $\SU(3)$ system \eqref{eq:F1}-\eqref{eq:D2}.  Note, in particular, that the deformation theory we refer to uses the anomaly cancellation condition in the form \eqref{eq:anomaly.def}, though we shall only use it in the form \eqref{eq:F2} in this article.
 \end{notation}

We now define one of the key objects of study in this section.
\begin{definition}
\label{def:Q} 
    Using Notation \ref{notation1},  define the complex vector bundle 
    $$Q = (T^{1,0}X)^* \oplus {\End} (E) \oplus T^{1,0}X,$$ 
    and define $Q^{0,p}=Q\otimes \Lambda^p(T^{0,1}X)^*$.  We denote sections of $Q^{0,p}$ by
\begin{equation}
\label{eq:s}
    s=\begin{bmatrix} 
        \kappa \\ \gamma \\ W 
    \end{bmatrix},
\end{equation}
where the rows  correspond to the direct sum decomposition of $Q$, written locally as
\begin{equation}
\label{eq:s.local}
      \kappa
      = \kappa_{j\bar{k_1}\ldots\bar{k_p}} dz^j\otimes d\bar{z}^{\bar{k_1}\ldots\bar{k_p}}, \quad \gamma=\gamma_{\bar{k_1}\ldots\bar{k_p}}\otimes d\bar{z}^{\bar{k_1}\ldots\bar{k_p}}, \quad W=W^j{}_{\bar{k_1}\ldots\bar{k_p}}\frac{\partial}{\partial z^j}\otimes d\bar{z}^{\bar{k_1}\ldots\bar{k_p}},
\end{equation}
where $d\bar{z}^{\bar{k_1}\ldots\bar{k_p}} =d\bar{z}^{\bar{k_1}}\wedge\ldots\wedge d\bar{z}^{\bar{k_p}}$. When there is no confusion, we will drop the indices $\bar{k_1},\ldots\bar{k_p}$ for ease of notation. Also for conciseness, we introduce the following `formal position  vectors' on $Q$:
\begin{equation}
\label{eq: vectors e_i}
    \be_1:= \begin{bmatrix} 
        \mathbf{1}_{(T^{1,0}X)^*} \\ 0 \\ 0 
    \end{bmatrix},
    \quad
    \be_2:= \begin{bmatrix} 
        0 \\ \mathbf{1}_{\End(E)} \\ 0 
    \end{bmatrix},
    \quad
    \be_3:= \begin{bmatrix} 
        0 \\ 0 \\ \mathbf{1}_{T^{1,0}X} 
    \end{bmatrix},
\end{equation}
so that
$$
s = \kappa\otimes\be_1 + \gamma\otimes\be_2 + W\otimes\be_3.
$$
NB.: In practice we will most often use $\be_1$.
\end{definition}

\begin{remark}
    In terms of moduli of the heterotic $\SU(3)$ system, the infinitesimal deformations will be parametrised by certain sections $s$ of $Q^{0,1}$ as in \eqref{eq:s}.  
    Here 
\begin{itemize}
    \item $\kappa$ is a (complexified) deformation of the Hermitian structure, or the metric, on $X$,
    \item $\gamma$ is a deformation of the  connection, or the Hermitian metric, on the vector bundle $E\rightarrow X$,
    \item  $W$ is the deformation of the complex structure of $X$, often referred to as the Beltrami differential.
\end{itemize} 
\end{remark}

To define the deformation operator $\bar D$ we are required to introduce two algebraic operators.

\begin{definition}
\label{def:F} 
    Recall the curvature $F$ of the connection on $E$ from Notation \ref{notation1}.  
    For $\gamma$ and $W$ as in \eqref{eq:s} we define $\sF\gamma\in \Gamma((T^{1,0}X)^*\otimes \Lambda^{p+1}(T^{0,1}X)^*)$ and $\sF W\in \Gamma(\End(E)\otimes \Lambda^{p+1}(T^{0,1}X)^*) $ via the local formulae 
\begin{equation}
\label{eq:sF}
    \sF \gamma 
    = \tr (F_{j \bar{k}} \gamma_{\bar{k_1}\ldots\bar{k_p}})dz^j \otimes d\bar{z}^{\bar{k}\bar{k_1}\ldots\bar{k_p}}, \quad \sF W = F_{j \bar{k}}W^j{}_{\bar{k_1}\ldots\bar{k_p}}\otimes d\bar{z}^{\bar{k}\bar{k_1}\ldots\bar{k_p}}\:,
\end{equation}
    which extend to globally defined sections on $X$.  These are natural actions of   $F$ on $\gamma$ and $W$.
\end{definition}

\begin{definition}
\label{def:T} 
    Recall $T$ given in Notation \ref{notation1}.  We define $\sT W\in\Gamma((T^{1,0}X)^*\otimes\Lambda^p(T^{0,1}X)^*)$  for $W$ as in \eqref{eq:s} by the local formula
\begin{equation}
\label{eq:sT}
      \sT W =T_{lj \bar{k}}W^l{}_{\bar{k_1}\ldots\bar{k_p}}dz^j\otimes d\bar{z}^{\bar{k}\bar{k_1}\ldots\bar{k_p}}\:,
\end{equation}
    which again extends globally on $X$.
\end{definition}

Beyond these algebraic operators, we also require a differential operator defined using an auxiliary connection which typically differs from the Chern connection.

\begin{definition}
\label{def:nabla+} 
    Recall the torsion $T$ from Notation \ref{notation1}; the \emph{Bismut connection} $\nabla^+$ on $T^{1,0}X$ is given locally by
\begin{equation}
\label{nabla+}
    \nabla^+_kW^j =\nabla_kW^j-T^j{}_{kl}W^l\:,
    \qforq
    W\in \Gamma(T^{1,0}X),
\end{equation}
    which again defines a global covariant derivative. We extend the Bismut connection to $T^{1,0}X\otimes\Lambda^p(T^{0,1}X)^*$ as 
    $$
    \nabla^+:= \nabla^+\otimes \mathbf{1}_{\Lambda^p(T^{0,1}X)^*} + \mathbf{1}_{T^{1,0}X}\otimes \nabla,
    \qforq   W\in\Gamma(T^{1,0}X\otimes\Lambda^p(T^{0,1}X)^*),
    $$
    ie. $\nabla^+$ acts as the Chern connection $\nabla$ on $(0,p)$-form indices of $W$.
\end{definition}

The connection $\nabla^+$ enables us to create a differential operator from the curvature $R$.
\begin{definition}
\label{def:Rnabla+} 
    Recall the curvature $R$ of the Chern connection of $(X,g)$, given in Notation \ref{notation1}, and the Bismut connection $\nabla^+$ given in Definition \ref{def:nabla+}.  For $W$ as in \eqref{eq:s}, we define $ \mathcal{R}\nabla^+ W\in\Gamma((T^{1,0}X)^*\otimes\Lambda^{p+1}(T^{0,1}X)^*)$ by the local formula
\begin{equation}
\label{eq:Rnabla+}
    \cR\nabla^+ W  
    = R_{\bar{k}j}{}^l{}_m\nabla^+_lW^m{}_{\bar{k_1}\ldots\bar{k_p}} dz^j\otimes d\bar{z}^{\bar{k}\bar{k_1}\ldots\bar{k_p}}\:.
\end{equation}
\end{definition}

%%%%%%%%%%%
\subsection{The operator \texorpdfstring{$\bar{D}$}{barD}}

With all of these definitions in place, we now define our deformation operator.

\begin{definition}
\label{def:Dbar} 
    The deformation operator $\bar{D}: \Gamma(Q^{0,p})  \rightarrow \Gamma(Q^{0,p+1}) $ is defined with respect to the splitting as in \eqref{eq:s} by
\begin{equation}
\label{eq: Dbar}
    \bar{D} = 
    \begin{bmatrix}
        \delbar & \alpha' \sF & \sT + \alpha'  \cR \nabla^+  \\
        0 & \delbar_E & \sF \\
        0 & 0 & \delbar
    \end{bmatrix},
\end{equation}
    where $\sF$ is given in Definition \ref{def:F}, $\sT$ is given in Definition \ref{def:T}, and $\cR\nabla^+$ is given in Definition \ref{def:Rnabla+}.
\end{definition}

\begin{remark}
\label{rmk:inf.defs.obs.1}
    It is shown in \cite{McOrist2022} that elements of $\ker\bar{D} \subset Q^{0,1}$ lead to infinitesimal deformations of heterotic $\SU(3)$ solutions. The moduli are then parametrised by the cohomology $H^{0,1}_{\bar{D}}(Q)$, 
    up to (complexified) symmetries of the system given by $\bar D$-exact terms.  We give a brief review of the infinitesimal deformations of heterotic $\SU(3)$ moduli in appendix \ref{app:Moduli}. Moreover, $\ker\bar{D}\subset Q^{0,2}$ gives potential obstructions to such deformations, and again it is the cohomology group $H^{0,2}_{\bar{D}}(Q)$ that parametrises these potential obstructions.  This motivates our interest in $\bar{D}$.
\end{remark}

We now show that $\bar{D}^2=0$ if, and only if, anomaly cancellation holds.  This enables us to define a cohomology theory associated to the operator $\bar{D}$, in the context of heterotic $\SU(3)$ solutions. The following proposition can also be found in \cite{McOrist:2024zdz}, but we reproduce it here, with our notation, for the reader's convenience.

\begin{prop}
\label{prop:Nilpotent}
    The operator $\bar{D}$ defined in \eqref{eq: Dbar} induces a differential complex if, and only if,  the anomaly cancellation condition \eqref{eq:F2} holds. 
\end{prop}
\begin{proof}
    For $s\in\Gamma(Q)$ as in \eqref{eq:s}, our goal is to show that the local coefficients of $\bar{D^2}s\in\Gamma(Q^{0,2})$ are 
\begin{align*}
     (\bar{D}^2 s)_{\bar{k} \bar{j}}
     &= \begin{bmatrix} \big( i \ddb \omega + \tfrac{\alpha'}{2}\tr (R \wedge R) - \tfrac{\alpha'}{2} \tr (F \wedge F)\big)_{m \bar{k} l \bar{j}} W^l dz^m \\ 0 \\0 \end{bmatrix} \\
     &=  (\big( i \ddb \omega + \tfrac{\alpha'}{2}\tr (R \wedge R) - \tfrac{\alpha'}{2} \tr (F \wedge F)\big)_{m \bar{k} l \bar{j}} W^l dz^m )\otimes\be_1
\end{align*}
     The computation for $s\in\Gamma(Q^{0,p})$ is similar for $p>0$, so we restrict to $p=0$ for convenience. By Definition \ref{def:Dbar}, we have
\begin{equation}
\label{eq: barD}
    \bar{D} s = \bar{D} 
    \begin{bmatrix} 
        \kappa \\ \gamma \\ W 
    \end{bmatrix} 
    = \begin{bmatrix}
        \delbar \kappa 
        + \alpha' \sF \gamma + \sT W + \alpha'  \cR \nabla^+  W \\ 
        \delbar_E \gamma + \sF W \\ \delbar W
    \end{bmatrix}
\end{equation}
    and so, using $\delbar^2 = 0$ and $\delbar_E^2=0$,
\begin{equation}
    \bar{D}^2 s 
    = \begin{bmatrix}
        ({\rm I}) + ({\rm II}) + ({\rm III}) \\ 
        ({\rm IV}) \\ 0 
    \end{bmatrix},
\end{equation}
    where
\begin{align*}
 ({\rm I}) 
    &= \alpha' \delbar (\sF \gamma) + \alpha' \sF \delbar_E \gamma,\\
    ({\rm II}) 
    &= \delbar (\sT W) + \sT \delbar W + \alpha' \sF (\sF W),\\
    ({\rm III}) 
    &= \alpha' \delbar (\cR \nabla^+ W) + \alpha' \cR \nabla^+ \delbar W,\\
    ({\rm IV}) 
    &= \delbar_E (\sF W) + \sF \delbar W.
\end{align*}
    
  We claim that terms $({\rm I})$ and $({\rm IV})$ are both zero by virtue of the Bianchi identity $\textrm{d}_AF=0$.  If we choose a local trivialisation where $\delbar_E=\delbar$, then since $A$ is the Chern connection it is given by a $(1,0)$-form, and  $F$ is of type $(1,1)$ so the Bianchi identity forces $\delbar F=0$.  Hence, for $({\rm I})$ we see that
\begin{align}
    \delbar_{\bar{l}} (\sF \gamma)_{j\bar{k}} + \sF (\delbar_{\bar{k}} \gamma)_{j\bar{l}} &= \delbar_{\bar{l}} \tr (F_{j \bar{k}} \gamma) +  \tr  (F_{j \bar{l}} \delbar_{\bar{k}} \gamma) \nonumber\\
    &=  
    \tr F_{j \bar{k} } \delbar_{\bar{l}} \gamma + \tr F_{j \bar{l} } \delbar_{\bar{k}} \gamma\:, \nonumber
\end{align}
   which is zero after antisymmetrisation in $j,k$. The same argument clearly applies to $({\rm IV})$.
    
    We can rewrite term $({\rm II})$ locally as
\begin{equation*}
  \big(\delbar_{\bar{k}} (T_{l m \bar{j}}W^l) +  T_{l m \bar{k}} \delbar_{\bar{j}} W^l + \alpha' \tr (F_{m \bar{k}} F_{l \bar{j}}W^l)\big)dz^m\otimes d\bar{z}^{\bar{k}\bar{j}}\:.
\end{equation*}
Recalling that $T=i\del\omega$ so $i\del\delbar\omega=-\delbar T$, after antisymmetrisation we are then left with
\begin{align*}
     (\bar{D}^2 s)_{\bar{k} \bar{j}} 
     &= \left(\big(\delbar_{\bar{k}} T_{lm \bar{j}} - \delbar_{\bar{j}} T_{lm \bar{k}} + \alpha' \tr (F_{m \bar{k}} F_{l \bar{j}}) - \alpha' \tr (F_{m \bar{j}} F_{l \bar{k}})\big) W^l dz^m + ({\rm III})_{\bar{k} \bar{j}} \right)\otimes\be_1\\
    &= \left( \big( i \ddb \omega -  \tfrac{\alpha'}{2} \tr (F \wedge F)\big)_{m \bar{k}l \bar{j}} W^l dz^m + ({\rm III})_{\bar{k} \bar{j}} \right) \otimes \mathbf{e}_1.
\end{align*}
    
    Finally, we expand term $({\rm III})$; using the fact that $R$ is a Riemann curvature tensor, and so the standard Bianchi identity gives $\delbar_{\bar{k}} R_{\bar{j} p}{}^\ell{}_m - (j \leftrightarrow k) =0$ (here and henceforth we use $(j \leftrightarrow k)$ to indicate the preceding expression with $j$ and $k$ exchanged):
\begin{align*}
    [({\rm III})_{\bar{k} \bar{j}}]_p 
    &= \alpha' \delbar_{\bar{k}} (R_{\bar{j} p}{}^\ell{}_m \nabla^+_\ell W^m) + \alpha' R_{\bar{k} p}{}^\ell{}_m \nabla^+_\ell \partial_{\bar{j}} W^m - (j\leftrightarrow k)
    \\&= \alpha' R_{\bar{j} p}{}^\ell{}_m ( \delbar_{\bar{k}} \nabla^+_\ell W^m - \nabla^+_\ell \delbar_{\bar{k}} W^m ) - (j \leftrightarrow k)\:.
\end{align*}
    Using  Definition \ref{def:nabla+} for the connection $\nabla^+$, this is
\begin{align*}
    [({\rm III})_{\bar{k} \bar{j}}]_p &= \alpha' R_{\bar{j} p}{}^\ell{}_i ( \delbar_{\bar{k}} \nabla_\ell W^m - \delbar_{\bar{k}} (T^i{}_{\ell q} W^q) - \nabla_\ell \delbar_{\bar{k}} W^m + T^m{}_{\ell q} \delbar_{\bar{k}} W^q) - (j \leftrightarrow k) \nonumber\\
    &= 
    \alpha' R_{\bar{j} p}{}^\ell{}_m ( \delbar_{\bar{k}} \nabla_\ell W^m - \nabla_\ell \delbar_{\bar{k}} W^m - \delbar_{\bar{k}} T^m{}_{\ell q} W^q)  - (j \leftrightarrow k) \nonumber\\
    &= 
    \alpha' R_{\bar{j} p}{}^\ell{}_m ( - R_{\bar{k} \ell}{}^m{}_q W^q - \delbar_{\bar{k}} T^m{}_{\ell q} W^q)  - (j \leftrightarrow k)\:.
\end{align*}
    We now use the following identity for the Chern connection in complex (non-K\"ahler) geometry,
\begin{equation}
\label{eq:Chern.R.sym}
    R_{\bar{k} q}{}^m{}_\ell 
    = R_{\bar{k} \ell}{}^m{}_q + \delbar_{\bar{k}} T^m{}_{\ell q},
\end{equation}
    which can be seen by direct computation using \eqref{R-conv}. We deduce that
\begin{equation*}
    [({\rm III})_{\bar{k} \bar{j}}]_p 
    = - \alpha' R_{\bar{j} p}{}^\ell{}_m R_{\bar{k} q}{}^m{}_\ell W^q - (j \leftrightarrow k)
\end{equation*}
and so, as a $1$-form,
\begin{align}
    ({\rm III})_{\bar{k} \bar{j}} &= 
    \alpha'  (R_{\bar{k} p}{}^\ell{}_m R_{\bar{j} q}{}^m{}_\ell - R_{\bar{j} p}{}^\ell{}_m R_{\bar{k} q}{}^m{}_\ell) W^q dz^p \nonumber\\
    &=\tfrac{\alpha'}{2}\big(\tr (R \wedge R)\big)_{\bar{k} p \bar{j} q} W^q dz^p\:.
\end{align}

  Altogether,
\begin{equation}
    (\bar{D}^2 s)_{\bar{k} \bar{j}} 
    = \left( \big( i \ddb \omega + \tfrac{\alpha'}{2}\tr(R \wedge R) - \tfrac{\alpha'}{2} \tr( F \wedge F)\big)_{m \bar{k} l \bar{j}} W^l dz^m \right) \otimes \be_1
\end{equation}
    and we see that $\bar{D}^2 s = 0$ is equivalent to $i \ddb \omega + \tfrac{\alpha'}{2} \left(\tr( R \wedge R) - \tr(F \wedge F) \right)=0$: this is precisely the anomaly cancellation condition \eqref{eq:F2} when the Chern connection is used on the tangent bundle.
\end{proof}    
From the proof of Proposition \ref{prop:Nilpotent}, we also get the following useful corollary.
\begin{corollary}
\label{cor:commute}
    Let $W=W^j\in\Omega^{0,p}(T^{1,0}X)$  and let $R=R_j{}^k{}_l\in\Omega^{0,1}\left((T^{1,0}X)^*\otimes\End(T^{1,0}X)\right)$ be the curvature of the Chern connection. Then
    \begin{equation}
        \delbar\left(\nabla_l^+W^k\right)=R_j{}^k{}_lW^j+\nabla^+_l\delbar W^k\:.
    \end{equation}
\end{corollary}
\begin{proof}
    This can be shown by direct computation, analogously following the computation of the $({\rm III})$-term in the proof of Proposition \ref{prop:Nilpotent}, using \eqref{eq:Chern.R.sym}.
\end{proof}

\begin{remark}
\label{rem:HYM2}
    As mentioned above, there is an alternative description of the heterotic $\SU(3)$ system, where in the anomaly cancellation condition \eqref{eq:F2} one uses a Hermitian Yang-Mills (HYM) connection on the tangent bundle instead of the Chern connection of $g$. To accommodate this choice, one needs to introduce a new metric $\tilde g$ on $TX$, whose Chern connection is now HYM. The construction then mimics the story above, except that we now use the Chern curvature of $\tilde g$ and the Bismut connection of $\tilde g$ when constructing the term $\cR \nabla^+$ in $\bar D$.
\end{remark}

%%%%%%%%
\subsection{The adjoint operator \texorpdfstring{$\bar{D}^*$}{barD*}}

To compute the adjoint of $\bar D$, we first need to find explicitly the natural metrics on $Q$ and $Q^{0,1}=Q\otimes (T^{0,1}X)^*$. We will refer frequently to objects given in Notation \ref{notation1}.

Let us express the metric on $Q$ induced by the Riemannian metric $g$ on $X$ and the Hermitian metric $h$ on the bundle $E$. The metric on $T^{1,0}X$ is given by $g=g_{jk}dz^jdz^k$. Similarly, $(T^{1,0}X)^*$ has metric $g^*$ with $g^*=g^{jk}\frac{\partial}{\partial z^j} \frac{\partial}{\partial z^k}$, where $g^{jl}g_{lk}=\delta^j_k$. On $\End(E)$, the induced Hermitian metric $h_{\End}$ is given for  $\beta,\gamma\in \End(E)$ with local components  $\beta=\beta^j_k e_j\otimes e^k$ and $\gamma=\gamma^j_k e_j\otimes e^k$ by $$h_{\End}(\beta,\gamma)=h_{jk}h^{lm}\beta^j_l\gamma^k_m,$$
where again $h^{jk}$ denotes the components of the inverse of $(h_{jk})$.   Thus the metric $g_{Q}$ on $Q$ is given by
\begin{equation}
    g_Q=\begin{bmatrix}
        g^*&0&0\\
        0&h_{\End}&0\\
        0&0&g
    \end{bmatrix}.
\end{equation}
Finally, we   extend $g_{Q}$ to a metric on $Q^{0,1}$  using $g^*$ with local coefficients $g^{\bar i \bar j}$ on $(T^{0,1}X)^*$.

Let $s\in \Gamma(Q),\ t\in \Gamma(Q)\otimes\Omega^{0,1}$  with local expressions
\begin{equation*}
     s=\begin{bmatrix}\theta\\\beta\\ V\end{bmatrix}=\begin{bmatrix}
         \theta_a dz^a\\
         \beta^b{}_c e_b\otimes e^c\\
         V^d \frac{\partial}{\partial z^d}
     \end{bmatrix}\quad\text{and}\quad
     t=\begin{bmatrix}
         \kappa\\ \gamma \\ W
     \end{bmatrix}=\begin{bmatrix}
         \kappa_{j\bar k}dz^j\\
         \gamma^l{}_{m\bar k} e_l\otimes e^m\\
         W^p{}_{\bar{k}} \frac{\partial}{\partial z^p}
    \end{bmatrix}\otimes d\bar{z}^{\bar{k}}.
\end{equation*}
We want to find $\bar D^*:\Gamma(Q^{0,1})\to\Gamma(Q)$ such that 
$$\int_X g_Q(\bar Ds,t)=\int_X g_Q(s,\bar D^*t)
$$ when $s,t$ have compact support, i.e.~so that $\bar{D}^*$ is the $L^2$ adjoint of $\bar{D}$. Using \eqref{eq: Dbar} and the formulae \eqref{eq:sF}, \eqref{eq:sT} and \eqref{eq:Rnabla+}:
\begin{align*}
\label{your_label_here}
\bar{D}s
  &=\begin{bmatrix} 
      \delbar\theta + \alpha' \sF \beta + \sT V + \alpha' \cR \nabla^+ V \\
      \delbar_E \beta + \sF V \\
      \delbar V
    \end{bmatrix} \\
  &=\begin{bmatrix}
      \big( \delbar_{\bar{q}} \theta_a + \alpha' F_{a\bar{q}}{}^c{}_b \beta^b{}_c + T_{da\bar{q}} V^d 
            + \alpha' R_{\bar{q}a}{}^r{}_d \nabla^+_r V^d \big) dz^a \\[2pt]
      \big( \delbar_{\bar{k}} \beta^b{}_c + F_{d\bar{k}}{}^b{}_{c} V^d \big) e_b \otimes e^c \\[2pt]
      \delbar_{\bar{q}} V^d \frac{\partial}{\partial z^d}
    \end{bmatrix} 
    \otimes d\bar{z}^{\bar{q}}.
\end{align*}By definition, the $L^2$ adjoints of $\delbar$ and $\delbar_E$ are $\delbar^*$ and $\delbar_E^*$ respectively.  Since the operators $\sF$ in \eqref{eq:sF} and $\sT$ in \eqref{eq:sT} are linear,  their pointwise (and hence $L^2$) adjoints are also linear maps that we can denote $\sF^*$ and $\sT^*$.  To compute $\sF^*$ and $\sT^*$ we see first that
\begin{equation}
\label{eq:sF*.1}
    g(\sF\beta,\kappa)
    = \beta^b{}_cF_{a\bar{q}}{}^c{}_b \kappa_{j\bar{k}}g^{aj}g^{\bar{q}\bar{k}}  
\end{equation}
and, if we write $\sF^*\kappa \in \Gamma(\End(E))$ locally as 
\begin{equation}
\label{eq:sF*.2}
    \sF^*\kappa = (F^*)^{j\bar{k}l}{}_m \kappa_{j\bar{k}} e_l \otimes e^m,
\end{equation}
then
\begin{equation}
\label{eq:sF*.3}
    h_{\End}(\beta,\sF^*\kappa)
    = \beta^b{}_c (F^*)^{j\bar{k}l}{}_m \kappa_{j\bar{k}} h_{bl} h^{cm}
\end{equation}
Comparing \eqref{eq:sF*.1} and \eqref{eq:sF*.3} we deduce that
\begin{equation}
\label{eq:sF*.4}
    (F^*)^{j\bar{k}l}{}_m
    = g^{aj}g^{\bar{q}\bar{k}}F_{a\bar{q}}{}^c{}_b h^{bl} h_{cm},
\end{equation}
which is equivalent to taking the usual adjoint (i.e. Hermitian transpose) on the endomorphism indices, and raising the 2-form indices using $g$. Similarly, we can define $\sF^*\gamma\in\Gamma(T^{1,0}X)$ locally by
\begin{equation}
\label{eq:sF*.5}
    \sF^*\gamma
    = g^{dp}g^{\bar{q}\bar{k}}F_{d\bar{q}}{}^b{}_c h_{bl} h^{cm} \gamma^{l}{}_{m\bar{k}} \frac{\partial}{\partial z_p}.
\end{equation}
and see that
\begin{equation*}
    g(V,\sF^*\gamma)=h_{\End}(\sF V,\gamma).
\end{equation*}
A similar calculation shows that if we define $\sT^*\kappa\in\Gamma(T^{1,0}X)$ by
\begin{equation}
\label{eq:sT*}
  \sT^*\kappa= g^{dp}g^{aj}g^{\bar{q}\bar{k}}T_{da\bar{q}}\kappa_{j\bar{k}}\frac{\partial}{\partial z^p}  
\end{equation}
then
\begin{equation*}
g(V,\sT^*\kappa)=g^*(\sT V,\kappa).
\end{equation*}

We are then left with computing the adjoint of $\cR\nabla^+$.   Using integration by parts for $\nabla$ and the fact that $\nabla$ is a metric connection, we see that when $V,\kappa$ are compactly supported we have
\begin{align*}
\int_Xg^*(\cR\nabla^+V,\kappa)&=\int_Xg^{aj}g^{\bar{k}\bar{q}}R_{\bar{q}a}{}^r{}_d(\nabla^+_rV^d)\kappa_{j\bar{k}}\\
&=\int_X \nabla_r(g^{aj}g^{\bar{k}\bar{q}}R_{\bar{q}a}{}^r{}_dV^d\kappa_{j\bar{k}})-\int_X \nabla_r(g^{aj}g^{\bar{k}\bar{q}}R_{\bar{q}a}{}^r{}_d\kappa_{j\bar{k}})V^d\\&\quad-\int_X g^{aj}g^{\bar{k}\bar{q}}R_{\bar{q}a}{}^r{}_d{T^d}_{rs}V^s\kappa_{j\bar{k}}\\
&=-\int_XV^dg^{aj}g^{\bar{k}\bar{q}}\big(\nabla_r( R_{\bar{q}a}{}^r{}_d\kappa_{j\bar{k}})+{T^s}_{rd}R_{\bar{q}a}{}^r{}_s\kappa_{j\bar{k}}\big)\\
&=-\int_X g\left(V^d\frac{\partial}{\partial z^d},g^{aj}g^{bp}g^{\bar{k}\bar{q}}\big(\nabla_r(R_{\bar{q}a}{}^r{}_b\kappa_{j\bar{k}})+{T^s}_{rb}R_{\bar{q}a}{}^r{}_s\kappa_{j\bar{k}}\big)\frac{\partial}{\partial z^p}\right).
\end{align*}
Therefore, if we define $(\nabla^+)^*\cR^*\kappa\in\Gamma(T^{1,0}X)$  by the local formula
\begin{equation}\label{eq:nabla+*R}
   (\nabla^+)^*\cR^*\kappa=-g^{aj}g^{bp}g^{\bar{k}\bar{q}}\big(\nabla_r(R_{\bar{q}a}{}^r{}_b\kappa_{j\bar{k}})+{T^s}_{rb}R_{\bar{q}a}{}^r{}_s\kappa_{j\bar{k}}\big)\frac{\partial}{\partial z^p}, 
\end{equation}
our observations combine into the following result. 

\begin{prop}\label{prop:adjoint}
    The adjoint $\bar{D}^*:\Gamma(Q^{0,1})\to\Gamma(Q)$ of the deformation operator $\bar{D}:\Gamma(Q)\to\Gamma(Q^{0,1})$ given in Definition \ref{def:Dbar} is 
\begin{equation}
\label{eq: barD*}
    \bar D^*
    =\begin{bmatrix}
    \delbar^*&0&0\\
    \alpha'\sF^*&\delbar_E^*&0\\
    \sT^*+\alpha'(\nabla^+)^*\cR^* &\sF^*&\delbar^*
    \end{bmatrix},
\end{equation}
    where $\sF^*$ is given by \eqref{eq:sF*.2}, \eqref{eq:sF*.4} and \eqref{eq:sF*.5}, $\sT^*$ is given by \eqref{eq:sT*}, and  $(\nabla^+)^*\cR^*$ is given by \eqref{eq:nabla+*R}.
\end{prop}
    
\begin{comment}
\begin{remark}
    The  reader might wonder whether the divergence term involving the Riemann curvature $\alpha'\nabla^+_rR_{\bar{q}a}{}^r{}_b$, which appears when expanding \eqref{eq:nabla+*R}, might vanish, in particular as this would simplify the calculation of the adjoint considerably. A short computation reveals that not to be the case in general, but it is however true up to terms of higher order in $\alpha'$.     
\end{remark}
\end{comment}

\subsection{Ellipticity} Now that we have formulae \eqref{eq: barD} for $\bar{D}$ and \eqref{eq: barD*} for $\bar{D}^*$, we can compute the (principal) symbol of $\bar{D}+\bar{D}^*:\Gamma(Q^{0,1})\to \Gamma(Q^{0,2}\oplus Q)$, so as to deduce its ellipticity.

\begin{prop}
\label{prop:elliptic}
    The symbol of    $\bar{D}+\bar{D}^*:\Gamma(Q^{0,1})\to \Gamma(Q^{0,2}\oplus Q)$ is given by its action on a section $s$ of $Q^{0,1}$ as in \eqref{eq:s} given a cotangent vector $\xi$ by the local formula 
\begin{align}
    \sigma_{\xi}(\bar{D}+\bar{D}^*)(s)
    &= \sigma_{\xi}(\bar{D}+\bar{D}^*)
    \begin{bmatrix} 
        \kappa \\ \gamma \\ W 
    \end{bmatrix} \nonumber \\
    &= \begin{bmatrix}        \big(\xi_{\bar{k}}\kappa_{j\bar{l}} + \alpha'R_{\bar{k}j}{}^m{}_n\xi_mW^n{}_{\bar{l}}\big)dz^j\otimes d\bar{z}^{\bar{k}\bar{l}} + \xi_{\bar{k}}\kappa_{j\bar{k}}dz^j \\        \xi_{\bar{k}}\gamma_{\bar{l}}\otimes d\bar{z}^{\bar{k}{l}} + \xi_{\bar{k}}\gamma_{\bar{k}} 
    \\        \xi_{\bar{k}}W^j{}_{\bar{l}}\frac{\partial}{\partial z^j}\otimes d\bar{z}^{\bar{k}{l}} + \big(-\alpha'R_{\bar{l}j}{}^m{}_n\xi_m\kappa_{j\bar{l}} + \xi_{\bar{k}}W^j{}_{\bar{k}}\big)\frac{\partial}{\partial z^j}
    \end{bmatrix}.
    \label{eq:symbol}
\end{align}
    Hence, if $R$ is bounded by $R_0$ and $\alpha'/R_0$ is sufficiently small, then $\ker\sigma_{\xi}(\bar{D}+\bar{D}^*)=0$ for $\xi\neq 0$, and thus $\bar{D}+\bar{D}^*$ is overdetermined elliptic.
\end{prop}

\begin{proof}
The computation of the symbol in \eqref{eq:symbol} is straightforward given the formulae for $\bar{D}$ and $\bar{D}^*$.  One point to note is that $\nabla^+$ and $\nabla$ only differ by an algebraic term in the torsion $T$, and so their symbols agree.  A similar consideration shows why the connection $A$ does not appear in the symbol.

Note that \eqref{eq:symbol} agrees with the symbol of $\delbar+\delbar^*$ when $\alpha'=0$, and so it will be injective for $\xi\neq 0$ in that case.  Since the symbol is linear in $\xi$, we can restrict now to $|\xi|=1$.  The set of injective linear maps is open, so if we have a bound $R_0$ on $R$, then we can choose $\alpha'$ sufficiently small so that the terms involving $\alpha'$ in the symbol are negligible, for any $\xi$ with $|\xi|=1$.  The result then follows.
\end{proof}

Proposition \ref{prop:elliptic} enables us, in particular, to infer that the kernel of $\bar{D}+\bar{D}^*$ is finite-dimensional.  Moreover, by Hodge theory (cf.~\cite{Wells1980}*{Theorem IV.5.2}) for the elliptic complex associated to $\bar{D}$ acting on $Q^{0,p}$,  we can deduce the following result.

\begin{theorem}
\label{prop:Hodge} 
    In the context of Notation \ref{notation1}, if $X$ is compact and $\alpha'$ is sufficiently small relative to a bound for the curvature $R$, then
\begin{equation*}
    H^{0,p}_{\bar{D}}(Q)\cong \{s\in Q^{0,p}\,:\bar{D}s=0\;\,
    \text{and}\;\,\bar{D}^*s=0\}.
\end{equation*}    
\end{theorem}
Theorem \ref{prop:Hodge} may prove useful, in the future, to compute the cohomology groups governing the deformation theory for the heterotic $\SU(3)$ system in concrete instances. See \S\ref{sec:Iwasawa} for a concrete example of this in the case of the Iwasawa manifold.

%\newpage
%%%%%%%%
\section{Serre duality}
In this section we will prove a version of Serre duality using homological algebra techniques. 
This will, in particular, prove Theorem \ref{thm:index}, which in turns yields our main result in Theorem \ref{thm:main}.  As mentioned above, this has important implications for the heterotic moduli problem, potentially leading to new invariants.  

\subsection{Double extensions and exact sequences} 
Recall the operator $\bar{D}$ given in Definition \ref{def:Dbar}: 
\begin{equation}
\label{eq:BarD2}
\bar{D}
=\left[\begin{array}{lll}
    \delbar & \alpha'\sF  & \sE \\ 
    0 & \delbar_E & \sF 
    \\ 0 & 0 & \delbar
    \end{array}\right] 
    \quad
    \begin{gathered} 
        (T^{1,0}X)^* \\ 
        \End (E)\\
        T^{1,0}X
    \end{gathered}\:,
\end{equation}
where 
\begin{equation}\label{eq:sE}
    \sE=\sT+\alpha'\cR\nabla^+: \Omega^{0,p}\left(T^{1,0}X\right) \rightarrow \Omega^{0, p+1}\left((T^{1,0}X)^*\right).
\end{equation}
As mentioned in \S\ref{ss:double.extension}, the differential operator \eqref{eq:BarD2} gives rise to two double extensions \eqref{eq:Atiyah} and \eqref{eq:Atiyah.dual} which enable us to study $H^{0,p}_{\bar{D}}(Q)$.
The first sequence in \eqref{eq:Atiyah} is the Atiyah extension \cite{Atiyah1957}, which yields the short exact sequence
\begin{equation}
\label{eq:SES1}
\xymatrix{
    0 \ar[r] & \Omega_{\delbar}^{0,p}(\End (E)) \ar[r] & \Omega_{\bar{D}_{1}}^{0,p}(Q_{1}) \ar[r] & \Omega_{\delbar}^{0,p}(T^{1,0}X) \ar[r] & 0\:,
}    
\end{equation}
where
\begin{equation}
\label{eq:barD1}
    \bar{D}_{1}
    = \left[\begin{array}{cc}
        \delbar_E & \sF  \\ 
        0 & \delbar
    \end{array}\right] \quad 
    \begin{gathered}    
        \End (E) \\ 
        T^{1,0}X
    \end{gathered}\:.
\end{equation}
Now note that we may write
\begin{equation}
\label{eq:Dbar.split}
    \bar{D}
    = \left[\begin{array}{ll}
        \delbar & \sH  \\ 
        0 & \bar{D}_1
    \end{array}\right]
    \quad 
    \begin{gathered}
        (T^{1,0}X)^* \\ 
        Q_1
    \end{gathered}\:,
\end{equation}
where 
\begin{equation}
\label{eq:sH}
    \sH\left[\begin{array}{c} 
        \gamma\\ 
        W 
    \end{array}\right]
    = \alpha' \sF \gamma + \sE W = \alpha' \sF \gamma + \sT W + \alpha' \cR \nabla^+W.
\end{equation}
The operator \eqref{eq:Dbar.split} gives a second extension as in \eqref{eq:Atiyah}, which yields the short exact sequence:
\begin{equation}
\label{eq:SES2}
    \xymatrix{
        0 \ar[r] & \Omega_{\delbar}^{0,p}\left((T^{1,0}X)^*\right) 
        \ar[r] & \Omega_{\bar{D}}^{0,p}(Q) \ar[r] & 
        \Omega_{\bar{D}_{1}}^{0,p}\left(Q_{1}\right) \ar[r] & 0
    }\:.
\end{equation}
Though the extension map of this sequence is non-tensorial, the maps in this sequence are the standard inclusion and projection of the sub-bundles of $Q=T^{1,0}X\oplus Q_1$ as a topological sum of bundles, respectively.

The dual Atiyah sequence (the first extension in \eqref{eq:Atiyah.dual})  leads to the short exact sequence
\begin{equation}
\label{eq:DualSES1}
\xymatrix{
    0 \ar[r] & \Omega_{\delbar}^{0,p}\left((T^{1,0}X)^*\right) \ar[r] & \Omega_{\bar{D}_{2}}^{0,p}(Q^*_{1}) \ar[r] & \Omega_{\delbar}^{0,p}(\End (E)) \ar[r] & 0\:,
}    
\end{equation}
where 
\begin{equation}
\label{eq:barD2}
    \bar{D}_{2}
=\left[\begin{array}{cc}
    \delbar & \alpha'\sF  \\ 
    0 & \delbar_E
\end{array}\right] \quad 
\begin{gathered}  (T^{1,0}X)^*  \\ 
     \End (E)
\end{gathered}\:.
\end{equation}  
It should be noted that, in order to view $\bar{D}_2$ as the correct dual operator on $Q_1^*$, we have to define $Q_1^*$ as the dual of $Q_1$ with respect to the (non-degenerate) pointwise pairing
\begin{equation}
\label{eq:pairing}
    \left\langle\left(\begin{array}{l}
    \beta \\ 
    V 
    \end{array}\right),
    \left(\begin{array}{l}
    \kappa \\ 
    \gamma 
    \end{array}\right)\right\rangle
    =V\lrcorner\kappa-\alpha'\tr(\beta\gamma)\:,
\end{equation}
where $(\beta,V)\in Q_1$ and $(\kappa,\gamma)\in Q_1^*$.  We may then write $\bar{D}$ as
\begin{equation}
\label{eq:barD2.dual}
    \bar{D}
    =\left[\begin{array}{ll}
    \bar{D}_2 & \sH^*  \\ 
    0 & \delbar\end{array}\right]
\quad 
\begin{gathered}
   Q^*_1 \\ 
   T^{1,0}X 
\end{gathered}\:,
\end{equation}
where  
\begin{equation}
\label{eq:sH*}
\sH^*(W)=\left[\begin{array}{cc} \sE W \\ \sF W\end{array}\right].
\end{equation}
We will see later, in the proof of Theorem \ref{prop:general.Serre}, that $\sH^*$ can be interpreted as a dual of $\sH$ in \eqref{eq:sH}.  Given $\bar D$ in \eqref{eq:barD2.dual}, we may then define $Q$ as an extension of $Q_1^*$ over the holomorphic tangent bundle as in \eqref{eq:Atiyah.dual}, which leads to the short exact sequence
\begin{equation}
\label{eq:DualSES2}
    \xymatrix{
        0 \ar[r] & \Omega_{\bar{D}_{2}}^{0,p}\left(Q_1^*\right) 
        \ar[r] & \Omega_{\bar{D}}^{0,p}(Q) \ar[r] & 
        \Omega_{\delbar}^{0,p}\left(T^{1,0}X\right) \ar[r] & 0
    }
\end{equation}

The short exact sequences \eqref{eq:SES1}, \eqref{eq:SES2}, \eqref{eq:DualSES1} and \eqref{eq:DualSES2} will be crucial tools in our proof of the analogue of Serre duality for $\bar{D}$.

\subsection{Simplified case}
We start in the context of Notation \ref{notation1}, so $\dim_{\C}(X)=3$, and that our original bundles, $\End (E), T^{1,0}X, (T^*X)^{1,0}$, have no holomorphic sections. The latter is true e.g.~when $E$ is (normalised and) stable and, in particular for the holomorphic tangent and cotangent bundles, when $X$ admits a Calabi--Yau metric.  In this simplified setting we prove our first version of Serre duality.  

\begin{prop}
\label{prop:Serre.simple}
    In the setting of Notation \ref{notation1}, suppose $X^3$ is a compact complex $3$-fold and that the holomorphic tangent, cotangent and gauge-endomorphism bundles on $X$ have no (non-trivial) holomorphic sections.  Then
 \begin{equation}
     H^{0,p}_{\bar D}(Q)\cong H^{0,3-p}_{\bar D}(Q)\:.
 \end{equation}
\end{prop}
\begin{proof}  
    To avoid cluttering, we shall use the superscript $p$ to indicate $(0,p)$ for our cohomology groups, we simply write $T$ for $T^{1,0}X$ and $T^*$ for $(T^*X)^{1,0}$, and we drop the subscripts indicating the operators defining the cohomology groups.  Recall that $X$ has holomorphic $(3,0)$-form $\Omega$, and we let $\partial_{\nabla}$ denote the Dolbeault operator associated to the connection on the respective holomorphic bundles.

    We first observe that the statement is trivially true for $p=0$ by the assumption of no holomorphic sections, which forces $H^0(Q)=H^3(Q)=0$, hence also $H^0(Q_1)=H^3(Q_1)=0$, cf.~\eqref{eq:Atiyah}.  Therefore we need only consider $p=1$. We now notice that the sequence \eqref{eq:SES2}, together with the assumption about absence of holomorphic sections, gives the following long exact sequence in cohomology, which will be the key tool for the remainder of the proof:
\begin{equation}
    \vcenter{
    \xymatrix{
    0 \ar[r] & H^1(T^{*}) \ar[r] & H^1(Q) \ar[r] & H^1(Q_{1}) \ar[dll]_-{\sH} \\
    & H^2(T^{*}) \ar[r] & H^2(Q) \ar[r] & H^2(Q_{1}) \ar[dll]_-{\sH} \\
    & H^3(T^{*}) \ar[r] & \ldots
    }
    }.
    \label{eq:LES1}
\end{equation}

Our first goal is to show that $H^3(T^*)$ vanishes. This can also be seen by ordinary Serre duality, as $H^0(T)\cong H^3(T^*)$. Note that $[\alpha] \in H^{3}\left(T^{*}\right)$ implies that $\alpha=\bar{\Omega} \otimes \kappa$ for some $\kappa\in \Omega^0\left(T^{*}\right)$. Moreover, if we choose $\alpha$ to be the harmonic representative of the class $[\alpha]$, then $$0=\delbar^{*} \alpha=-* \partial_{\nabla} * \alpha,$$ where $*$ is the complex linear extension of the real Hodge star operator (which is not sesquilinear). We deduce that 
$$
\del_{\nabla}\kappa=0.
$$
Complex-conjugating and raising the index on $\kappa$ to view it as a $W\in \Omega^0(T)$, we see that 
$\delbar W=0$, and hence 
$W\in H^0(T)=\{0\}$.   
Therefore 
$\kappa=0$, and thus $[\alpha]=0$. 
We conclude that
\begin{equation}
\label{eq:H3.zero}
    H^3(T^*)=0.
\end{equation}

Given \eqref{eq:H3.zero}, the exact sequence \eqref{eq:LES1} becomes:
\begin{equation}
    \vcenter{
    \xymatrix{
    0 \ar[r] & H^1(T^{*}) \ar[r] & H^1(Q) \ar[r] & H^1(Q_{1}) \ar[dll]_-{\sH} \\
    & H^2(T^{*}) \ar[r] & H^2(Q) \ar[r] & H^2(Q_{1}) \ar[r] & 0
    }
    }.
    \label{eq:LES1b}
\end{equation}
By exactness of \eqref{eq:LES1b}, we have 
\begin{equation}
\label{eq:LES.iso.1}
\begin{aligned}
    & H^1(Q) \cong H^1\left(T^{*}\right) \oplus \ker (\sH ) ,\\
    & H^{2}(Q) \cong H^{2}\left(Q_{1}\right) \oplus \frac{H^{2}\left(T^{*}\right)}{\im (\sH )}.
\end{aligned}
\end{equation}

Since $\End (E)$ is self-dual and we have assumed it has no holomorphic sections, standard Serre duality implies $$H^{3}(\End (E))=H^{0}(\End(E))=0\quad\text{and}\quad H^1(\End(E))\cong H^2(\End(E)).$$
This fact, together the assumption that $H^0(T)=0$ means that the long exact sequence which follows from \eqref{eq:SES1} becomes:
\begin{equation}
    \vcenter{
    \xymatrix{
    0 \ar[r] & H^1(\End (E)) \ar[r] & H^1(Q_{1}) \ar[r] & H^1(T) \ar[dll]_-{\sF} \\
    & H^2(\End (E)) \ar[r] & H^2(Q_{1}) \ar[r] & H^2(T) \ar[r] & 0
    }
    }
    \label{eq:LES1c}
\end{equation}
By exactness of \eqref{eq:LES1c},
\begin{equation}
\label{eq:LES.iso.2}
\begin{aligned}
    & H^1(Q_1) \cong H^1\left(\End(E)\right) \oplus \ker (\sF )\cong  H^2\left(\End(E)\right) \oplus \ker (\sF ),\\
    & H^{2}\left(Q_{1}\right) \cong H^{2}(T) \oplus \frac{H^{2}(\End (E))}{\im (\sF )}.
\end{aligned}
\end{equation}
Combining \eqref{eq:LES.iso.1} and \eqref{eq:LES.iso.2} we see that
$$
\begin{aligned}
    H^{2}(Q) 
    &\cong H^{2}\left(Q_{1}\right) \oplus \frac{H^{2}\left(T^{*}\right)}{\im (\sH )} \\
    & \cong H^{2}(T) \oplus \frac{H^{2}(\End (E))}{\im (\sF )} \oplus \frac{H^{2}\left(T^{*}\right)}{\im (\sH )}.
 \end{aligned}
 $$
 Therefore,
\begin{equation}
\label{eq:h2Q} 
    h^{2}(Q)
    =h^{2}(T) +h^{2}(\End (E))+h^{2}\left(T^{*}\right)  -|\im (\sF )|-|\im  \mathscr{H}|.
\end{equation}

Now we observe that
\begin{equation}
\label{eq:h2T}
    h^{2}(T) =h^{2,2} =h^{1,1} =h^{1}\left(T^{*}\right),
\end{equation} where the first equality uses contraction by $\Omega$ and the second equality is Hodge duality. We also have 
\begin{equation}
\label{eq:h2T*}
    h^{2}(T^{*}) =h^{1,2} =h^{2,1} =h^{1}(T),
\end{equation} where the middle equality is complex conjugation.  Note that \eqref{eq:h2T} and \eqref{eq:h2T*} are just consequences of ordinary Serre duality. Returning to \eqref{eq:h2Q} and using \eqref{eq:LES.iso.1}, \eqref{eq:LES.iso.2}, \eqref{eq:h2T} and \eqref{eq:h2T*}, we conclude: 
\begin{align*}
    h^{2}(Q)
    &=h^{1}\left(T^{*}\right)  +h^{2}(\End (E))+\overset{|\ker (\sF )|}{\overbrace{h^{1}(T) -|\im (\sF )|} }
     -|\im (\sH )|\\
    &= h^{1}\left(T^{*}\right)+\underbrace{h^{2}(\End (E))+|\ker(\sF)|}_{h^{1}\left(Q_{1}\right)} -|\im (\sH )|\\
&=h^{1}\left(T^{*}\right)+h^{1}\left(Q_{1}\right) -|\im (\sH )| \\
    &= h^{1}\left(T^{*}\right)+|\ker (\sH )|=h^1(Q).
    \qedhere
\end{align*}
\end{proof}

\begin{remark}
    The proof of Proposition \ref{prop:Serre.simple}  does not rely on the details of the map $\sH$, so it is also true for the moduli problems of \cites{Anderson2014, Ossa2014a}, with spurious modes included. By contrast, our proof of Serre duality in greater  generality (Theorem \ref{prop:general.Serre} below) does rely on the specificities of  $\sH$.
\end{remark}

%%%%%%%%%%%
\subsection{General case}
\label{sec:GenSerre}

We now move to a more general setting, where we (i) relax the assumption on complex dimension; (ii) assume a looser form of the anomaly cancellation condition; and (iii) allow the various bundle components of $Q$ to have holomorphic sections. 

Specifically, on a complex $n$-fold $X$ with holomorphic volume $(n,0)$-form $\Omega$, Hermitian metric $g$, and holomorphic vector bundle $E\to X$, we ask that  
\begin{equation}
\label{eq:anomaly.modified}
    \delbar T=\tfrac{\alpha'}{2}\left(\tr  R \wedge R - \tr F \wedge F\right),
\end{equation}
for \emph{some}  $T\in \Omega^{2,1}(X)$ which need \emph{not} be the torsion, and $R$ and $F$ are as in Notation \ref{notation1}.   
By the proof of Proposition \ref{prop:Nilpotent}, the corresponding nilpotent operator $\bar D$ in \eqref{eq:BarD2} may be defined, now using this tensor $T$ instead: 

\begin{theorem}
\label{prop:general.Serre}
    Let $X^n$ be a compact complex manifold endowed with a holomorphic volume $(n,0)$-form $\Omega$, Hermitian metric $g$, and a holomorphic vector bundle $E\to X$ with Chern connection $A$ satisfying \eqref{eq:anomaly.modified}.  Let $\bar{D}$ be defined by \eqref{eq:BarD2}, where the operator $\sT$ in \eqref{eq:sE} is defined instead by $T$ in \eqref{eq:anomaly.modified}. Then the following notion of Serre duality holds:
\begin{equation}
     H^{0,p}_{\bar D}(Q)\cong H^{0,n-p}_{\bar D}(Q)\:.
\end{equation}
\end{theorem}

\begin{proof}
    As in the proof of Proposition \ref{prop:Serre.simple}, we simplify notation by using the superscript $p$ to denote $(0,p)$ for the cohomology groups, and we drop the subscripts denoting the differential operator defining the cohomology.  We also use $T$ for $T^{1,0}X$ and $T^*$ for $(T^{1,0}X)^*$.

    Using the long exact sequence derived from the dual extension \eqref{eq:DualSES2}, we can compute the dimension of $H^p(Q)$ by
\begin{equation}
\label{eq:first.ordinary.Serre}
    h^p(Q)=h^p(Q_1^*)+h^p(T)-|\im (\sH^*_p )|-|\im (\sH^*_{p-1} )|\:,
\end{equation}
    where $\sH^*_p$ acts on $T$-valued $(0,p)$-forms and is given in \eqref{eq:sH*}.  If we compute the dimension of $H^{n-p}(Q)$ using the original extension \eqref{eq:SES2}, we find that
\begin{align}
    h^{n-p}(Q) 
    &= h^{n-p}(Q_1) + h^{n-p}(T^*) - |\im (\sH_{n-p})| - |\im (\sH_{n-p-1})| \notag \\
    &= h^p(Q^*_1) + h^p(T) - |\im (\sH_{n-p})| - |\im (\sH_{n-p-1})|\:,
    \label{eq:second.ordinary.Serre}
\end{align}
    where $\sH_{q}$ acts $Q_1$-valued $(0,q)$-forms as in \eqref{eq:sH}. We have used ordinary Serre duality in the second equality in \eqref{eq:second.ordinary.Serre}, which gives $h^p(T)=h^{n-p}(T^*)$. It will then be clear that $h^p(Q)=h^{n-p}(Q)$ if we can show the following identities:
\begin{align}
\label{eq:H*H}
    |\im (\sH^*_p )|=|\im (\sH_{n-p-1} )|\quad\text{and}\quad
    |\im (\sH^*_{p-1} )|=|\im (\sH_{n-p} )|\:.
\end{align}
    Note that the second identity in \eqref{eq:H*H} follows from the first, by duality.

    To prove \eqref{eq:H*H}, as in ordinary Serre duality, we will use a non-degenerate pairing induced by the holomorphic $(n,0)$-form $\Omega$, which we denote by
\begin{equation}
\label{eq:pairing.2}
\begin{gathered}
    (\cdot,\cdot): \;H^q(Q_1)\times H^{n-q}(Q^*_1)\rightarrow\mathbb{C}\\
    \left(u,v\right) =\int_X\langle u,v\rangle\wedge\Omega\;,
\end{gathered}
\end{equation}
    where $\langle\cdot,\cdot\rangle$ is induced  by \eqref{eq:pairing}.  We also have the analogous pairing between $H^q(T)$ and $H^{n-q}(T^*)$,  for which we use the same symbol. We then suppose
\begin{equation}
\label{eq:w.z}
    w=\left[W\right]
    \in  H^p(T)\quad\text{and}\quad u=\left[\left(\begin{array}{c} \beta \\ V\end{array}\right)\right]\in H^{n-p-1}(Q_1).
\end{equation}
Our goal is to show that
\begin{equation}
\label{eq:Hdual}
    (u,\sH^*_p(w))
    =(-1)^{n-p}(\sH_{n-p-1}(u),w)\:.
\end{equation}
Since pairings appearing on each side are non-degenerate, this immediately yields \eqref{eq:H*H}.

Writing out the left-hand side of \eqref{eq:Hdual} using \eqref{eq:sH*} and \eqref{eq:pairing}, we have:
\begin{align}
    (u,\sH^*_p(w))
    &=\int_X \big(V\lrcorner(\sE W)-\alpha'\tr(\beta\wedge \sF W)\big)\wedge\Omega \notag\\
    &=\int_X \big(-\alpha'\tr(\beta\wedge\sF W)+V\lrcorner(\sT W)+\alpha'V\lrcorner(\cR\nabla^+ W)\big)\wedge\Omega.\label{eq:LHS} 
\end{align}
Note that the contractions in \eqref{eq:LHS} also involve wedge products of the forms, so the ordering of the $T$-valued and $T^*$-valued forms matters.
Similarly, the pairing on the right-hand side of \eqref{eq:Hdual} can be computed using \eqref{eq:sH}:
\begin{align}
    (\sH_{n-p-1}(u),w)&=\int_X (\alpha'\sF\beta+\sE V)(W)\wedge\Omega\notag\\
    &=\int_X\big(\alpha'\sF\beta(W)+(\sT V)(W)+\alpha'(\cR\nabla^+V)(W)\big)\wedge\Omega.\label{eq:RHS}
\end{align}
Here we have written the contraction of $T$-valued and $T^*$-valued forms differently to emphasise the ordering.
Using Definition \ref{def:F} of $\sF$ and Definition \ref{def:T} of $\sT$ we quickly see that
\begin{equation}
\label{eq:F.T.identities}
    \tr(\beta\wedge\sF W)
    =(-1)^{n-p-1}(\sF\beta)(W)\quad\text{and}\quad V\lrcorner(\sT W)=(-1)^{n-p}(\sT V)(W).
\end{equation}
Recall now Definition \ref{def:Rnabla+}, Corollary \ref{cor:commute} and the fact that $\delbar V=\delbar W=0$.  Using these observations, and suppressing local indices which do not play a role, we may compute
\begin{align}
\int_X V\lrcorner (\cR\nabla^+W)\wedge\Omega&=\int_X V^j\wedge R_j{}^l{}_m\nabla^+_lW^m\wedge\Omega\notag\\
&=(-1)^{n-p-1}\int_X R_j{}^l{}_mV^j\wedge \nabla^+_lW^m\wedge\Omega\notag\\
&=(-1)^{n-p-1}\int_X\delbar (\nabla^+_mV^l)\wedge\nabla^+_lW^m\wedge\Omega\notag\\
&=-(-1)^{2(n-p-1)}\int_X \nabla_m^+V^l\wedge\delbar(\nabla_l^+W^m)\wedge\Omega\notag\\
&=-\int_X\nabla_m^+V^l\wedge R_j{}^m{}_lW^j\wedge\Omega\notag\\
&=(-1)^{n-p}\int_X (\cR\nabla^+V)(W)\wedge\Omega.\label{eq:R.nabla+.identities}
\end{align}
Combining \eqref{eq:LHS}, \eqref{eq:RHS}, \eqref{eq:F.T.identities} and \eqref{eq:R.nabla+.identities}, we quickly see that \eqref{eq:Hdual} holds as claimed. 
Hence \eqref{eq:H*H} holds and thus, by \eqref{eq:first.ordinary.Serre} and \eqref{eq:second.ordinary.Serre} we have $h^p(Q)=h^{n-p}(Q)$. 
\end{proof}

Theorem \ref{thm:index} follows directly from Theorem \ref{prop:general.Serre}. Furthermore, we immediately conclude the following:
\begin{corollary}
    For manifolds $X$ of odd complex dimension $n$, the Euler characteristic of the complex vanishes, that is
    \begin{equation}
    \label{eq:Index}
        \chi(X,\bar{D}) := \sum_{k=0}^n (-1)^k \dim \left(H^{(0,k)}_{\bar{D}}(Q)\right)= 0\:.
    \end{equation}
    In particular, for the heterotic $\SU(3)$ system with $\alpha'$ sufficiently small, the operator $\bar{D}+\bar{D}^*$ has vanishing index. 
\end{corollary}
\begin{proof}
    The vanishing of the Euler characteristic is an immediate consequence of the Serre duality property of Theorem \ref{prop:general.Serre}. Using Theorem \ref{prop:Hodge}, we can then conclude that the index of $\bar{D}+\bar{D}^*$ vanishes. 
\end{proof}

\section{Examples}

In this section we describe two examples where our formalism can be applied, so that we can investigate the local properties of the moduli space of the heterotic $\SU(3)$ system in concrete terms.

\subsection{Calabi--Eckmann \texorpdfstring{$S^3\times S^3$}{S3xS3}}
Our first example 
%where our formalism can be applied. Although it 
is not strictly a heterotic $\SU(3)$ solution, but it leads to a solution to a generalisation of the heterotic $\SU(3)$ system known as the \emph{coupled $\SU(3)$-instanton} equation, cf.~\cites{GarciaFernandez2023,GarciaFernandez2023b,GarciaFernandez2024,daSilva2024}.  Nevertheless, our techniques still apply in this setting. 

\begin{remark} In the similar settings studied in the references above, several related results to those in this article were also established.  For example, an analogous operator to $\bar{D}$ was identified in \cite{GarciaFernandez2024}*{p.31}.    
\end{remark}

We begin with the definition of the ambient complex manifold.

\begin{definition}
\label{def:CE.cx}
    A natural action of $\mathbb{C}$ on $(\mathbb{C}^2\setminus\{0\})\times (\mathbb{C}^2\setminus\{0\})$ is given by:
\begin{equation}
\label{eq:CE.action}
    \gamma\cdot \big((z_1,z_2),(w_1,w_2)\big)
    =\big(e^\gamma(z_1,z_2), e^{i\gamma}(w_1,w_2)\big).
\end{equation}
The action \eqref{eq:CE.action} defines the Calabi--Eckmann complex structure $J$ on the quotient:
\begin{equation}
X^6=  \big((\mathbb{C}^2\setminus\{0\})\times (\mathbb{C}^2\setminus\{0\})\big)/\mathbb{C}\cong S^3\times S^3.
\end{equation}
\end{definition}

The Calabi--Eckmann complex structure in Definition \ref{def:CE.cx} is not compatible with even a balanced Hermitian structure (and thus no K\"ahler structure), but there is a natural compatible Hermitian form:  
\begin{definition}
\label{def:CE.omega}
    If we let $\pi_j:S^3\to\mathbb{CP}^1$, for $j=1,2$, denote the Hopf fibration on each factor in $X$, we can let $\alpha_j$ denote the associated connection $1$-form on $S^3$ for this fibration satisfying
\begin{equation}
    \rd\alpha_j
    = \pi_j^*\omega_{\mathbb{CP}^1},
\end{equation}
where $\omega_{\mathbb{CP}^1}$ is the K\"ahler form for the Fubini--Study metric on $\mathbb{CP}^1$.  We may then define a Hermitian form $\omega$ on $(X^6,J)$ in Definition \ref{def:CE.cx} by:
\begin{equation}
\label{eq:CE.omega}
  \omega
  =\rd\alpha_1 +\rd\alpha_2 +\alpha_1\wedge\alpha_2.
\end{equation}
This Hermitian form endows $X^6\cong S^3\times S^3$ with the product metric consisting of round metrics on the factors.
\end{definition}

It is useful here to record the torsion and properties of the Hermitian structure.

\begin{lemma}
\label{lem:CE.omega}
    The Hermitian structure $\omega$ on $(X^6,J)$ given in Definition \ref{def:CE.omega} satisfies
    \begin{equation}
    \label{eq:CE.T}
    T=i\delbar\omega= \frac{1}{2}(\alpha_1+i\alpha_2)\wedge(\rd\alpha_1-i\rd\alpha_2)\neq 0
\end{equation}
and
\begin{equation}
\label{eq:CE.PC}
    \del\delbar\omega=0, %=\rd\rd^c\omega.
\end{equation}
so it is pluriclosed.  Moreover,
\begin{equation}
\label{eq:CE.domega2}
    \rd\omega^2
    =(\alpha_2-\alpha_1) \wedge\omega^2
\end{equation}
and 
\begin{equation}
    \rd(\alpha_2-\alpha_1)\neq 0,
\end{equation}
so the Hermitian structure is not even locally conformally balanced.
\end{lemma}

\begin{proof}
  We note from \eqref{eq:CE.omega} that
\begin{equation}
\label{eq:CE.domega}
    \rd\omega
    =\alpha_2\wedge\rd\alpha_1 -\alpha_1\wedge\rd\alpha_2.
\end{equation}
We see immediately from \eqref{eq:CE.omega} that the torsion $T$ is given by \eqref{eq:CE.T}. We deduce that $\rd T=0$ and thus $\del T=0$, so \eqref{eq:CE.PC} holds. Equation \eqref{eq:CE.domega2} follows from \eqref{eq:CE.omega} and \eqref{eq:CE.domega}, which completes the proof.
\end{proof}

\begin{remark}
    Notice that \eqref{eq:CE.PC} implies that the anomaly cancellation \eqref{eq:F2} can only be satisfied if the right-hand side is zero, i.e.~$\tr(R\wedge R)=\tr(F\wedge F)$.
\end{remark}

There is a natural nowhere-vanishing $(3,0)$-form $\Omega$, so that $(\omega,\Omega)$ is an $\SU(3)$-structure, as follows.

\begin{definition}
\label{def:CE.Omega}
    On  $(\mathbb{C}^2\setminus\{0\})\times (\mathbb{C}^2\setminus\{0\})$ we define
    \begin{equation}
\Omega=\frac{\big((z_1\rd z_2-z_2\rd z_1)+i(w_1\rd w_2-w_2\rd w_1)\big)\wedge (\rd z_1\wedge \rd z_2+\rd w_1\wedge \rd w_2)}{(|z_1|^2+|z_2|^2)(|w_1|^2+|w_2|^2)}.
\end{equation}
This $(3,0)$-form is invariant under the action \eqref{eq:CE.action}, and so it passes to the quotient $(X^6,J)$ as a nowhere-vanishing form.    It is straightforward to calculate that
\begin{equation}
\label{eq:CE.dOmega}
    \rd\Omega=(\alpha_2-\alpha_1)\wedge\Omega.
\end{equation}
This relation reflects the fact that the Calabi--Eckmann complex structure $J$ is integrable.
\end{definition}

\begin{remark}
    Altogether we see, from  \eqref{eq:CE.dOmega} and \eqref{eq:CE.domega2}, that the equations \eqref{eq:F1} and \eqref{eq:D2} from the heterotic $\SU(3)$ system are satisfied by the $\SU(3)$-structure $(\omega,\Omega)$ on $(X^6,J)$, modulo $\Omega$ and $\omega^2$ respectively.  
\end{remark}  

%%%%%%%%%%%

We now turn to the gauge-theoretic aspects of the heterotic $\SU(3)$ system in this setting.

\begin{prop}
\label{prop:CE.Chern}  
    Let $(X^6,J)$ be as in Definition \ref{def:CE.cx}. The Bismut and Hull connections on $T^{1,0}X$ are gauge-equivalent flat connections.  The Chern connection on $T^{1,0}X$ is not flat, but its curvature $R$ satisfies
\begin{equation}
\label{eq:CE.trR}
    \tr(R\wedge R)=0.
\end{equation}
\end{prop}

\begin{proof}
We let $\{e_1,e_2,e_3\}$, $\{e_4,e_5,e_6\}$ be  standard left-invariant coframes on $\SU(2)\cong S^3$, so that
\begin{equation}
    \rd e_i=\frac{1}{2}\epsilon_{ijk} e_j\wedge e_k,
    \qforq \{i,j,k\}=\{1,2,3\},
\end{equation}
and a similar equation for $\{i,j,k\}=\{4,5,6\}$.  We can choose these coframes to be orthonormal for the metric $g$ on $X$, and we can choose $e_1=\alpha_1$ and $e_4=\alpha_2$ in the notation of Definition \ref{def:CE.omega}. This means that
\begin{equation}
    \rd\alpha_1=e_2\wedge e_3\quad\text{and}\quad \rd\alpha_2=e_5\wedge e_6,
\end{equation}
from which one can compute the matrix for the Levi-Civita connection: 
\begin{equation}
\label{eq:Gamma.LC}
  \Gamma^{g}=\frac{1}{2}\left[\begin{array}{cc} \Gamma_1 & 0 \\
  0 & \Gamma_2\end{array}\right],
\end{equation}
where
\begin{equation}
    \Gamma_1=\begin{bmatrix}
       0 & e_3 & -e_2 \\
       -e_3 & 0 & e_1 \\
       e_2 & -e_1 & 0
    \end{bmatrix}
\end{equation}
and a similar formula holds for $\Gamma_2$ involving $\{e_4,e_5,e_6\}$.  Recall that the Bismut connection $\nabla^+$ is related to the Levi-Civita connection $\nabla^g$ by
\begin{equation}
    \nabla^+
    =\nabla^g-\frac{1}{2}g^{-1}\rd^c\omega.
\end{equation}
We see that
\begin{equation}
    \rd^c\omega
    =\alpha_1\wedge\rd\alpha_1 +\alpha_2\wedge\rd\alpha_2
    =e_1\wedge e_2\wedge e_3+e_4\wedge e_5\wedge e_6.
\end{equation}
Hence, the Bismut connection $\nabla^+$ is trivial and flat as claimed.  
The curvature of the Hull and Bismut connections are related by $\del\delbar\omega$, which vanishes by \eqref{eq:CE.PC}, so the Hull connection is also flat.  Since $X$ is simply connected, the Hull and Bismut connections are then gauge-equivalent.

We now recall that the Chern connection $\nabla$ on $T^{1,0}X$ is related to the Levi-Civita connection $\nabla^g$ by
\begin{equation}\label{eq:Chern.LC}
    g(\nabla_XY,Z)=g(\nabla^g_XY,Z)-\frac{1}{2}\rd\omega(JX,Y,Z).
\end{equation}
If we use the basis for the $(1,0)$-forms on $X$ given by $\{e_1+ie_4,e_2+ie_3,e_5+ie_6\}$, \eqref{eq:CE.domega}, \eqref{eq:Gamma.LC}, \eqref{eq:Chern.LC} we see that the connection matrix $\Gamma$ for the Chern connection is
\begin{equation}\label{eq:Gamma}
    \Gamma=\frac{1}{2}\begin{bmatrix}
        0 & e_3+ie_3 & -e_5+ie_6 \\ 
        -e_3+ie_2 & -2ie_1 & 0 \\
        e_5+ie_6 & 0 & -2ie_4
    \end{bmatrix}.
\end{equation}
Using \eqref{eq:Gamma} one may compute that the curvature $R$ of $\nabla$ is given by
\begin{equation}\label{eq:CE.R}
    R=\frac{1}{4}\begin{bmatrix}
        -2ie_2\wedge e_3-2ie_5\wedge e_6 & 0 & 0\\
        0 & -2ie_2\wedge e_3 & (-e_3+ie_2)\wedge (-e_5+ie_6)\\
        0 & -(e_3+ie_2)\wedge (e_5+ie_6) & -2ie_5\wedge e_6
    \end{bmatrix}.
\end{equation}
Clearly, the Chern connection is not flat and one may calculate from \eqref{eq:CE.R} that \eqref{eq:CE.trR} holds as claimed.
\end{proof}

\begin{remark}\label{rmk:CE.barD}
    Lemma \ref{lem:CE.omega} and Proposition \ref{prop:CE.Chern} show that if we choose any flat connection $A$ on a holomorphic vector bundle $E\to X$ given in Definition \ref{def:CE.cx}, then the anomaly cancellation condition \eqref{eq:F2} and the instanton condition \eqref{eq:D1} are satisfied.  Hence, we can build the operator $\bar{D}$ as in \eqref{eq: barD}, which is a differential coboundary by Proposition \ref{prop:Nilpotent}, with cohomology groups $H^{0,p}_{\bar{D}}(Q)$.  Note that the operators $\sT$ and $\cR\nabla^+$ appearing in $\bar{D}$ are both nontrivial, even if the operator $\sF$  is zero.
\end{remark}

%%%%%%%%%%%

We may now combine our results in this section and deduce the following result.

\begin{theorem}
\label{thm:CE}
    Let $(X^6\cong S^3\times S^3,J)$ be the Calabi--Eckmann complex manifold given in Definition \ref{def:CE.cx}, endowed with the $\SU(3)$-structure $(\omega,\Omega)$ given in Definitions \ref{def:CE.omega} and \ref{def:CE.Omega}.  Let $E=T^{1,0}X$ and let $A$ be the Bismut (or Hull) connection on $E$.  These data solve the anomaly cancellation \eqref{eq:F2} and instanton conditions \eqref{eq:D1} in the heterotic $\SU(3)$ system, and solve the equations \eqref{eq:F1} and \eqref{eq:D2}, modulo $\Omega$ and $\omega^2$ respectively.  
    Moreover, the cohomology groups $H^{0,1}_{\bar{D}}(Q)$ and $H^{0,2}_{\bar{D}}(Q)$ are well-defined and trivial.
 \end{theorem}

\begin{proof}
  The first part of the statement follows from Lemma \ref{lem:CE.omega}, equation  \eqref{eq:CE.dOmega} and Proposition \ref{prop:CE.Chern}.  

  For the second part of the statement, we first observe that the cohomology groups are well-defined by Remark \ref{rmk:CE.barD}.  We now note that, since the connection we chose on $E$ is flat, the operator $\bar{D}_1$ given in \eqref{eq:barD1} defined on the bundle $Q_1$ in \eqref{eq:SES1} is just the usual $\delbar$-operator.  Note also that all of the bundles we are considering in the long exact sequence \eqref{eq:LES1} are trivial, so their ordinary Dolbeault cohomology is determined by that of $X$.  However, $H^1(X)=H^2(X)=0$ as $X\cong S^3\times S^3$, so the Dolbeault cohomology in these degrees for $T^*$ and $Q_1$ vanish.  Therefore, the terms $H^j(T^*)$ and $H^j(Q_1)$ for $j=1,2$ in \eqref{eq:LES1} are zero, which means that   $H^{0,1}_{\bar{D}}(Q)=H^{0,2}_{\bar{D}}(Q)=0$ by \eqref{eq:LES1}.
\end{proof}

\begin{remark}
    We can interpret Theorem \ref{thm:CE} as saying that the solution of the modified heterotic $\SU(3)$ system given here, with Calabi--Eckmann $X\cong S^3\times S^3$ as its base manifold, is rigid and unobstructed.  Therefore, the moduli space in this setting is locally just an isolated point.
\end{remark}

%%%%%%%%%%%

\subsection{Iwasawa manifold} 
\label{sec:Iwasawa}
We now consider a solution to the heterotic $\SU(3)$ system which is strongly motivated by work in \cite{LopesCardoso2002}. It should be noted that this solution is also non-physical, in the sense that it requires $\alpha'<0$. Nevertheless, it provides a good toy model for our results.  

We begin by describing the complex manifold on which we solve the system.

\begin{definition}
\label{dfn:Iwasawa}
    Let $G$ be the complex Heisenberg group and let $\Gamma\subseteq G $ be a lattice, as follows:
\begin{equation}
\label{eq:Heisenberg}
 G=\left\{
 \begin{bmatrix}
     1 & z_1 & z_3 \\
     0 & 1 & z_2 \\
     0 & 0 & 1
\end{bmatrix}\,:\,z_1,z_2,z_3\in\mathbb{C}\right\}    \qandq \Gamma=\left\{
 \begin{bmatrix}
     1 & z_1 & z_3 \\
     0 & 1 & z_2 \\
     0 & 0 & 1
\end{bmatrix}\,:\,z_1,z_2,z_3\in\mathbb{Z}\oplus i\mathbb{Z}\right\}.
\end{equation}
    The variables $z_1,z_2,z_3$ define complex coordinates on $G$ and hence a complex structure $J$ on $X^6=G/\Gamma$, which is called the \emph{Iwasawa manifold}.  
\end{definition}

\begin{remark}
   We see that $G$ in \eqref{eq:Heisenberg} is nilpotent, and thus $X$ is a nilmanifold endowed with the left-invariant complex structure $J$.  Such complex manifolds, including their deformations of complex structure, have been well-studied and have a rich theory. In fact, the Iwasawa manifold is a canonical and important example in this context.
\end{remark}

We can now define a natural left-invariant $\SU(3)$-structure on the Iwasawa manifold $(X,J)$.

\begin{definition}\label{dfn:Iwasawa.SU3} Let $(X,J)$ be as in Definition \ref{dfn:Iwasawa}. 
We define a global (left-invariant) basis for $\Omega^{1,0}(X)$ by:
\begin{equation}\label{eq:Iwasawa.coframe}
    \alpha_1=\rd z_1,\quad 
    \alpha_2=\rd z_2,\quad \alpha_3=-\rd z_3+z_1\rd z_2,
\end{equation}
which satisfy
\begin{equation} \label{eq:Iwasawa.coframe.diff} \rd\alpha_1=0,\quad\rd\alpha_2=0,\quad \rd\alpha_3=\alpha_1\wedge\alpha_2.  
\end{equation}
We may then define a Hermitian form $\omega$ by
\begin{equation}
\label{eq:Iwasawa.omega}
    \omega=\frac{i}{2}(\alpha_1\wedge\overline{\alpha_1}+
\alpha_2\wedge\overline{\alpha_2}+
\alpha_3\wedge\overline{\alpha_3}).
\end{equation}
With respect to the complex structure $J$, the compatible metric $g$ on $X$ is simply
\begin{equation}
\label{eq:Iwasawa.g}
    g=\alpha_1\overline{\alpha_1} +\alpha_2\overline{\alpha_2} +\alpha_3\overline{\alpha_3}.
\end{equation}
We finally define a nowhere vanishing $(3,0)$-form $\Omega$ by
\begin{equation}\label{eq:Iwasawa.Omega}
\Omega=\alpha_1\wedge\alpha_2\wedge\alpha_3.
\end{equation}
Note that $|\Omega|_{\omega}$ is constant, by \eqref{eq:Iwasawa.g}.
\end{definition}

Given this $\SU(3)$-structure, we may now study its torsion, and its relation to  the heterotic $\SU(3)$ system.

\begin{lemma}\label{lem:Iwasawa.torsion}
    The $\SU(3)$-structure in Definition \ref{dfn:Iwasawa.SU3} satisfies
    \begin{equation}\label{eq:Iwasawa.torsion}
        \rd\Omega=0,\quad T=i\del\omega=-\frac{1}{2}\alpha_1\wedge\alpha_2\wedge\overline{\alpha_3},\quad 2i\del\delbar\omega=\alpha_1\wedge\alpha_2\wedge\overline{\alpha_1}\wedge\overline{\alpha_2}.
    \end{equation}
\end{lemma}

\begin{proof}
 The first equation in \eqref{eq:Iwasawa.torsion} clearly follows from \eqref{eq:Iwasawa.coframe.diff} and \eqref{eq:Iwasawa.Omega}.  We then note from \eqref{eq:Iwasawa.coframe.diff} and \eqref{eq:Iwasawa.omega} that
 \begin{equation*}
     \rd\omega=\frac{i}{2}(\alpha_1\wedge\alpha_2\wedge\overline{\alpha_3}-\alpha_3\wedge \overline{\alpha_1}\wedge\overline{\alpha_2}).
 \end{equation*}
 The formula for $T$ in \eqref{eq:Iwasawa.torsion} then follows.  We then calculate from \eqref{eq:Iwasawa.coframe.diff} that
 \begin{equation*}
     \rd T=-\frac{1}{2}\alpha_1\wedge\alpha_2\wedge\overline{\alpha_1}\wedge\overline{\alpha_2},
 \end{equation*}
 which yields the final equation in \eqref{eq:Iwasawa.torsion}.
\end{proof}

\begin{remark} Lemma \ref{lem:Iwasawa.torsion} shows that \eqref{eq:F1} and \eqref{eq:D1} in the heterotic $\SU(3)$ system are satisfied.
\end{remark}

Before turning to the gauge theory aspects of the heterotic $\SU(3)$ system, 
we note the following well-known fact (see e.g.~\cite{DiScalaVezzoni}*{\S4.1}), which one may straightforwardly deduce from the formulae in Definitions \ref{dfn:Iwasawa} and \ref{dfn:Iwasawa.SU3}.

\begin{lemma}
\label{lem:Iwasawa.R}
    The Chern connection $\nabla$ arising from the $\SU(3)$-structure in Definition \ref{dfn:Iwasawa.SU3} is flat, i.e.~it has curvature $R=0$. 
\end{lemma}

\begin{remark}
    In \cite{LopesCardoso2002}, the Bismut connection is adopted in the anomaly condition, rather than the Chern connection, but this also turns out to be flat.
\end{remark}

We may now define the connection we desire, which is a variation on the connection used in \cite{LopesCardoso2002}.

\begin{definition}
\label{dfn:Iwasawa.A} 
    In the context of Definitions \ref{dfn:Iwasawa} and \ref{dfn:Iwasawa.SU3}, define a connection $A$ on a rank 2 holomorphic vector bundle $E\to X$ so that it has curvature given by   
    \begin{equation}
    \label{eq:Iwasawa.F}
    F = \frac{i}{4}
    \begin{bmatrix}
        \alpha_1\wedge\overline{\alpha_1} -\alpha_2\wedge\overline{\alpha_2} & 0\\
        0 &  -\alpha_1\wedge\overline{\alpha_1} + \alpha_2\wedge\overline{\alpha_2}
    \end{bmatrix}.
    \end{equation}
    Note that $F$ takes values in a $\mathfrak{u}(1)\subset \mathfrak{su}(2)$ subalgebra, so we can view $E$ as having  structure group $\SU(2)$.
\end{definition}

We may now show that the remaining equations \eqref{eq:F2} and \eqref{eq:D2} in the heterotic $\SU(3)$ system are satisfied, given Lemmas \ref{lem:Iwasawa.torsion} and \ref{lem:Iwasawa.R}.

\begin{lemma}\label{lem:Iwasawa.F}
The Hermitian form $\omega$ in \eqref{eq:Iwasawa.omega} and curvature $F$ of the connection $A$ in Definition \ref{dfn:Iwasawa.A} satisfy
\begin{equation}\label{eq:Iwasawa.F2D2}
F\wedge\omega\wedge\omega=0\quad\text{and}\quad 2i\del\delbar \omega=-4\tr(F\wedge F).
\end{equation}
\end{lemma}
\begin{proof}
    The first equation in \eqref{eq:Iwasawa.F2D2} follows immediately from \eqref{eq:Iwasawa.F} and \eqref{eq:Iwasawa.omega}.  We may also calculate from \eqref{eq:Iwasawa.F} that
    \begin{equation*}
        \tr(F\wedge F)=\frac{1}{4}\alpha_1\wedge\overline{\alpha_1}\wedge\alpha_2\wedge\overline{\alpha_2}.
    \end{equation*}
    This, combined with Lemma \ref{lem:Iwasawa.torsion}, yields the second equation in \eqref{eq:Iwasawa.F2D2}.
\end{proof}

In summary, the heterotic $\SU(3)$ system is satisfied with $\alpha'=-4$, which is nonzero, but neither small nor positive.  This means the solution is ``non-physical'' but still of mathematical interest.  We conclude with our main result in this subsection.  

\begin{theorem}
\label{thm:Iwasawa}
    Let $(X^6 = G/\Gamma, J)$ be the Iwasawa manifold given in Definition \ref{dfn:Iwasawa}, endowed with the $\SU(3)$-structure $(\omega, \Omega)$ given in Definition \ref{dfn:Iwasawa.SU3}. Let $E\to X$ be the rank 2 holomorphic vector bundle with connection $A$ and curvature $F$ as given in Definition \ref{dfn:Iwasawa.A}. These data solve the anomaly cancellation condition \eqref{eq:F2}, with $\alpha'=-4$, as well as the instanton condition \eqref{eq:D1} and equations \eqref{eq:F1} and \eqref{eq:D2}, i.e.~the full  heterotic $\SU(3)$ system.  Moreover, the cohomology groups $H^{0,1}_{\bar{D}}(Q)\cong H^{0,2}_{\bar{D}}(Q)$  are well-defined and have dimension 11.
\end{theorem}

\begin{proof}
The first part of the statement follows from Lemmas \ref{lem:Iwasawa.torsion}, \ref{lem:Iwasawa.R} and \ref{lem:Iwasawa.F}.  We then know that the cohomology groups are well-defined by Proposition \ref{prop:Nilpotent} and that they are isomorphic by Theorem \ref{prop:general.Serre}.  We are therefore only left with computing their dimension.

First, we note that we have a global holomorphic (co-)frame $\{\alpha_1,\alpha_2,\alpha_3\}$ for $(T^{1,0}X)^*$ in \eqref{eq:Iwasawa.coframe}, which in turn defines a global holomorphic frame $\{V_1,V_2,V_3\}$ for $T^{1,0}X$ (where $V_j$ is dual to $\alpha_j$), so these bundles are holomorphically trivial.      
Our interest is in the kernel of $\bar{D}$, from \eqref{eq: Dbar}, among $Q$-valued $(0,1)$-forms, and we have a basis of $\Omega^{0,1}(X)$ given by $\{\overline{\alpha_j}\}$, which satisfy
\begin{equation}
\label{eq:Iwasawa.coframe.diff.2}
\delbar\overline{\alpha_1}=0,  \quad\delbar\overline{\alpha_2}=0, \quad\delbar\overline{\alpha_3}= \overline{\alpha_1}\wedge\overline{\alpha_2}
\end{equation}
by \eqref{eq:Iwasawa.coframe.diff}. Moreover, $\overline{\alpha_1}$ and $\overline{\alpha_2}$ are not globally $\delbar$-exact, as one can see by the $\delbar$-cohomology of $X$,  e.g.~\cite{Angella}*{Table 3.1}. 
We deduce that any $Q$-valued $(0,1)$-form may be written as $s=(\kappa,\gamma,W)^{\rm T}$, locally of the form
\begin{equation}
\label{eq:Iwasawa.s}
    \kappa=\kappa_{j\bar{k}}\alpha_j\otimes \overline{\alpha_{\bar{k}}},\quad \gamma=\begin{bmatrix}
       i\gamma_{1\bar{k}} & -\gamma_{2\bar{k}}+i\gamma_{3\bar{k}}\\
       \gamma_{2\bar{k}}+i\gamma_{3\bar{k}} &-i\gamma_{1\bar{k}}
    \end{bmatrix}\otimes \overline{\alpha_{\bar{k}}},\quad W=W_{j\bar{k}}V_j\otimes\overline{\alpha_{\bar{k}}}.
\end{equation}

By Lemma \ref{lem:Iwasawa.R} we see that the operator $\bar{D}$ in \eqref{eq: Dbar} reduces to
\begin{equation}
    \bar{D}=\begin{bmatrix}
        \delbar & -4\sF & \sT \\ 0 & \delbar & \sF \\ 0 & 0 &\delbar
    \end{bmatrix},
\end{equation}
where we denote the holomorphic structure on $E$ as simply $\delbar$.  Hence $\bar{D}s=0$ if and only if
\begin{equation}
\label{eq:Iwasawa.Dbar.2}
    \delbar\kappa+\alpha'\sF\gamma+\sT W=0,\quad \delbar\gamma+\sF W=0,\quad \delbar W=0.
\end{equation}
Substituting \eqref{eq:Iwasawa.s} into the last equation in \eqref{eq:Iwasawa.Dbar.2} shows that
\begin{equation}
\label{eq:Iwasawa.W3}
    W_{j\bar{3}}=0.
\end{equation}
Instead substituting \eqref{eq:Iwasawa.s} into the first equation in \eqref{eq:Iwasawa.Dbar.2}, we see that $\delbar\kappa$ and $\sF \gamma$ can only be multiples of $\overline{\alpha_1}\wedge\overline{\alpha_2}$ by \eqref{eq:Iwasawa.coframe.diff.2} and \eqref{eq:Iwasawa.F}, whereas 
\begin{equation}
\label{eq:Iwasawa.TW}
    2\sT W = \left( -W_{1\bar{1}}\alpha_2 + W_{2\bar{1}}\alpha_1 \right) \otimes \overline{\alpha_3}\wedge\overline{\alpha_1}
    + \left( -W_{1\bar{2}}\alpha_2 + W_{2\bar{2}}\alpha_1 \right) \otimes \overline{\alpha_3}\wedge\overline{\alpha_2}.
\end{equation}
We deduce that \eqref{eq:Iwasawa.TW} vanishes which, together with \eqref{eq:Iwasawa.W3} allows us to deduce that
\begin{equation}
W=V_3\otimes(W_{3\bar{1}}\overline{\alpha_{\bar{1}}}+W_{3\bar{2}}\overline{\alpha_{\bar{2}}}).
\end{equation}
Note that $\sF W=0$ now holds automatically.

Hence, \eqref{eq:Iwasawa.Dbar.2} now reduces to
\begin{equation}\label{eq:Iwasawa.Dbar.3}
    \delbar\kappa-4\sF\gamma=0,\quad \delbar\gamma=0.
\end{equation}
The second equation in \eqref{eq:Iwasawa.Dbar.3} together with \eqref{eq:Iwasawa.s} gives
\begin{equation}
\label{eq:Iwasawa.gamma3}
    \gamma_{j\bar{3}}=0,
    \qforq j=1,2,3.
\end{equation}
Finally, the first equation in \eqref{eq:Iwasawa.Dbar.3} with \eqref{eq:Iwasawa.s} yields
\begin{equation}
\label{eq:Iwasawa.kappa3}
    \kappa_{j\bar{3}}\alpha_j+2\gamma_{1\bar{2}}\alpha_1+2\gamma_{1\bar{1}}\alpha_2=0.
\end{equation}
This determines each $\kappa_{j\bar{3}}$.
Overall, we have the following free parameters: $\kappa_{j\bar{k}}$, $\gamma_{j\bar{k}}$ and $W_{3\bar{k}}$ for $j=1,2,3$ and $k=1,2$.  In total these are 14 for elements of $\ker\bar{D}$ acting on $Q^{0,1}$.

We now wish to impose $\bar{D}^*s=0$.  By Proposition \ref{prop:adjoint} this is equivalent to
\begin{equation}\label{eq:Iwasawa.adjoint}
    \delbar^*\kappa=0,\quad -4\sF^*\kappa+\delbar^*\gamma=0,\quad \sT^*\kappa+\sF^*\gamma+\delbar^*W=0.
\end{equation}
We first notice that the differential terms in \eqref{eq:Iwasawa.adjoint} all vanish, since $\delbar^* \overline{\alpha_{\bar{k}}} =0$,
for $k=1,2,3$. 
We then see that
\begin{equation*}
    \sF^*\kappa=\frac{1}{4}\begin{bmatrix}
        i & 0 \\ 0 & -i
    \end{bmatrix}\otimes (\kappa_{1\bar{1}}-\kappa_{2\bar{2}}).
\end{equation*}
Since this must be zero, we have $\kappa_{1\bar{1}}=\kappa_{2\bar{2}}$.  
Using \eqref{eq:Iwasawa.kappa3}, we now compute
\begin{equation}
\label{eq:Iwasawa.T*}
    \sT^*\kappa=-\frac{1}{2}\kappa_{2\bar{3}}V_1+\frac{1}{2}\kappa_{1\bar{3}}V_2=\gamma_{1\bar{1}}V_1-\gamma_{1\bar{2}}V_2.
\end{equation}
We also have
\begin{equation}
\label{eq:Iwasawa.F*}
    \sF^*\gamma=-\frac{1}{2}\gamma_{1\bar{1}}V_1+\frac{1}{2}\gamma_{1\bar{2}}V_2.
\end{equation}
Combining \eqref{eq:Iwasawa.T*} and \eqref{eq:Iwasawa.F*} with \eqref{eq:Iwasawa.adjoint} and \eqref{eq:Iwasawa.gamma3}, we find that 
\begin{equation}
    \gamma_1=0.
\end{equation}
Notice that this means $\kappa_{j\bar{3}}=0$, by \eqref{eq:Iwasawa.kappa3}.

We conclude that the kernel of $\bar{D}+\bar{D}^*$ acting on $Q^{0,1}$ is a codimension $3$ subspace of $\ker\bar{D}$, which means that  it has dimension $11$.  Note that, even though $\alpha'$ is not assumed to be sufficiently small, since the $\cR\nabla^+$ term drops out of the operator $\bar D$ (as the Chern connection is flat) and so Proposition \ref{prop:Hodge} may be applied to deduce the result.  In fact, in this case the operator $\bar D$ defines an ordinary holomorphic structure, and so normal Hodge theory can be applied.
\end{proof}

  For comparison, it can be easily checked that the dimension of $H^{0,1}_{\delbar}(Q)$ with the diagonal complex structure $\delbar$ on $Q$ is $18$. This is thus an example where the cohomology $H^{0,1}_{\bar D}(Q)$ does not decompose into the usual cohomology groups of the individual bundles, as is the case e.g.~of the Standard Embedding \cites{candelas1985vacuum, Chisamanga:2024xbm}.

Note that though we have succeeded in computing the dimension of $H^{0,1}_{\bar{D}}(Q)\cong H^{0,2}_{\bar{D}}(Q)$ in the example in Theorem \ref{thm:Iwasawa}, it seems challenging to employ this direct approach more generally.  This further motivates the use of the isomorphism we establish, in the next section,  between the $\bar{D}$-cohomology groups and the \v{C}ech cohomology groups, which may be calculated using  algebraic methods related to sheaf cohomology.

\begin{remark}
    The example in Theorem \ref{thm:Iwasawa} shows that it is important to understand the relation between the infinitesimal deformations in $H^{0,1}_{\bar{D}}(Q)$ and the obstructions in $H^{0,2}_{\bar{D}}(Q)$.  Specifically, whether the obstructions are effective or not, and whether there are conditions which ensure that the obstructions are ineffective or, on the other extreme, are fully effective, for example in sufficiently generic situations as one might expect.
\end{remark}

%\newpage
%%%%%%%%
\section{\v{C}ech cohomology and a Dolbeault theorem}
\label{sec: Cech and Dolbeault}

In this section we show that a heterotic $\SU(3)$ solution defines a notion of \v{C}ech cohomology, and we prove (part of) a Dolbeault-type theorem, namely that the first \v{C}ech cohomology group is isomorphic to the first cohomology group defined by the deformation operator $\bar{D}$.
This more algebraic perspective on the cohomology associated to $\bar{D}$ should prove useful in understanding and computing the deformation and obstruction spaces for the heterotic $\SU(3)$ system. 

%%%%%%%%%%%
\subsection{Motivation}

By Proposition \ref{prop:Nilpotent}, if we satisfy the anomaly cancellation condition \eqref{eq:F2}, which is one of the heterotic F-term constraints, then $\bar D^2=0$ and we may define cohomology groups $H^{{0,p}}_{\bar{D}}(Q)$.  Since we are working on  holomorphic vector bundles, we are led to view $\bar{D}$ as a ``Dolbeault-type'' operator with that nilpotency as an integrability condition.  
  
A Dolbeault operator $\delbar_{V}$, or holomorphic structure (since $\delbar_V^2=0$), on a rank $r$ holomorphic bundle $\pi:V\to X$, can always be identified locally with the usual Dolbeault operator: for suitable open covers $\{U_i\}$ of $X$ and trivialisations $\phi_i:\pi^{-1}(U_i)\to U_i\times\C^r$, one has
\begin{equation*}
    \delbar_{V}|_{U_i} =\phi_i^{-1}\circ\delbar\circ \phi_i\:,
\end{equation*}
where $\delbar$ is the standard Dolbeault operator on $U_i\times\C^r$ arising from $X$ and $\C^r$. In a sense, we can view the $\phi_i$ as `local gauge transformations' which are well-adapted to the operator $\delbar_V$. Even though $\delbar_V$ cannot be identified with $\delbar$ globally in general, using such trivialisations one can prove a  Dolbeault-type theorem identifying Dolbeault cohomology and \v{C}ech cohomology of $V$, cf.~eg \cite{dolbeault1953}:
$$
H^{(0,p)}_{\delbar_V}(V)\cong\check{H}^p(V)\:.
$$

Now, while our operator $\bar D$ \emph{does not} in principle define a holomorphic structure on $Q$, we will show in this section that it nonetheless also admits such a local identification with $\delbar$:
\begin{equation*}
    \bar D|_{U_i}=\phi_i^{-1}\circ\delbar\circ \phi_i.
\end{equation*}
However, the `local gauge transformations' $\phi_i$ will also involve holomorphic derivatives, and hence so do the `transition functions'
\begin{equation*}
    \psi_{ij}=\phi_i^{-1}\circ \phi_j%: U_{i} \cap U_{j} \rightarrow U_{i} \cap U_{j},
    \quad\text{on}\quad
    U_i\cap U_j,
\end{equation*}
which we should now think of as more general `holomorphic transition operators'. 
We use these local trivialisations to define a \v{C}ech cohomology and prove our Dolbeault-type theorem for the operator $\bar D$.

\begin{remark}
    As well as helping towards the moduli problem for the heterotic $\SU(3)$ system, the perspective we take in this section may also, in the future, be used for understanding notions of stability and related ideas for the `holomorphic bundle' $(Q,\bar D)$.
\end{remark}

%%%%%%%%%%%
\subsection{Local trivialisation of \texorpdfstring{$\bar D$}{barD}}
\label{ss:local.triv.1}

In order not to clutter the discussion, let us begin by considering a simpler situation where $E$ is either altogether absent or,  equivalently, has trivial structure group. As we shall see in \S\ref{sec:CheckGauge}, the gauge bundle can be included in a straightforward manner. 

The bundle $Q$ in Definition \ref{def:Q} then simplifies to 
\begin{equation}
\label{eq:Q.simple}
    Q=(T^{1,0}X)^*\oplus T^{1,0}X
\end{equation}
and the operator $\bar{D}$ in Definition \ref{def:Dbar} reduces to 
\begin{equation}\label{eq:D.simple}
\bar{D}
=\left[\begin{array}{cc}
    \delbar & \sT+\alpha'\cR  \nabla^+ \\
    0 & \delbar
\end{array}\right]\:.
\end{equation}
This setting sufficiently includes the nuances of interest and, of course, we still have $\bar{D}^2=0$ due to \eqref{eq:F2} (cf. Proposition \ref{prop:Nilpotent}) which here becomes
\begin{equation}
\label{eq:anomaly.simple}
    i\del\delbar\omega
    =-\tfrac{\alpha'}{2}\tr(R\wedge R) \:.  
\end{equation}

Recall that $\nabla$ is the Chern connection of the metric $g$ on $X$ and so it satisfies $\nabla^{0,1}=\delbar$.  Thus, $\nabla$ and its curvature $R$ can be locally expressed, respectively, by a matrix of $(1,0)$-forms $\Gamma$
and a matrix of $(1,1)$-forms
\begin{equation}
\label{eq:delbar.Gamma}
    R=\delbar\Gamma\:.
\end{equation}
We can then also locally define  
\begin{equation}
    \Gamma\cdot\nabla^+W =\Gamma_a{}^c{}_b\nabla^+_cW^bdz^a.
\end{equation}

\begin{prop}
\label{prop: D bar (loc)= delbar}
    In the setting above, given any sufficiently small open chart $U$ on $X$, there exists a local section $\tau\in\Gamma(U,(T^{1,0}X)^*\otimes (T^{1,0}X)^*)$  such that the map 
\begin{equation}
\label{eq: phi matrix 2x2}
    \phi
    := \left[\begin{array}{cc}
        1 & \tilde\tau+\alpha'\Gamma \cdot \nabla^+ \\
        0 & 1
    \end{array}\right] \quad 
\begin{gathered}
   (T^{1,0}X)^* \\ 
   T^{1,0}X 
\end{gathered}\:
\end{equation}
   trivialises $\bar D$, in the following sense: 
\begin{equation}
    \bar{D}|_{U} =\phi^{-1}\circ\delbar\circ\phi.
\end{equation}
\end{prop}

\begin{proof}
    Recall that $T=i\del\omega$ in Notation \ref{notation1} and that $\del\delbar=-\delbar\del$. Hence, \eqref{eq:anomaly.simple} can be written as
\begin{equation}
\label{eq: delbar Tab}
   \delbar T=\tfrac{\alpha'}{2}\tr (R\wedge R)\:.
\end{equation} 
    On any local chart $U \subset X$, it follows from \eqref{eq:delbar.Gamma} and \eqref{eq: delbar Tab} that
$$
\delbar\big(T-\tfrac{\alpha'}{2}\tr(\Gamma\wedge\delbar\Gamma)\big)=0\:.$$
    Making $U$ smaller if necessary, we may apply the $\delbar$-Poincaré lemma to deduce that 
\begin{equation}
\label{eq:tau.def}
    T-\tfrac{\alpha'}{2}\tr(\Gamma\wedge\delbar\Gamma) = \delbar \tau\:,
\end{equation}
    for some $(2,0)$-form $\tau$. The expression \eqref{eq:tau.def} can be written in coordinates as
\begin{equation}\label{eq:tau.def.local}
    T_{ab}-\alpha'\tr(\Gamma_{[a}\delbar\Gamma_{b]}) =\delbar\tau_{ab}\:,
\end{equation}
    where the square brackets indicate skew-symmetrisation, and we have suppressed antiholomorphic indices. Defining  $\tilde{\tau}$ by
\begin{equation}
\label{eq:tildetau.def}
     \tilde\tau_{ab}
     :=\tau_{ab}-\tfrac{\alpha'}{2}\tr(\Gamma_a\Gamma_b)\:,
\end{equation}
    then  \eqref{eq:tau.def.local} gives 
\begin{equation}
\label{eq:DefTildeTau}
    \delbar\tilde{\tau}_{ab}=\delbar\left(\tau_{ab}-\tfrac{\alpha'}{2}\tr(\Gamma_a\Gamma_b)\right)
    = T_{ab}-\alpha'\tr(\Gamma_{a}\delbar\Gamma_{b})\:.
\end{equation}
    Note that $\tilde\tau$ is not skew-symmetric in its indices. 

    Consider next the action of the extension map $\sT+\cR\nabla^+$, for which we recall Definitions \ref{def:T} and \ref{def:Rnabla+}.  For $W \in \Omega^{0,p}\left(T^{1,0}\right)$, it is given locally on $U$ (suppressing local antiholomorphic indices) by: 
\begin{align}
    T_{a b} W^{b}+\alpha'R_{a}{}^{c}{}_{b} \nabla^+_{c} W^{b}
    &= T_{a b} W^{b}+\alpha'\delbar \Gamma_{a}{ }^{c}{ }_{b} \nabla^+_{c} W^{b}\notag \\
    &=T_{a b} W^{b} +\alpha'\delbar\left(\Gamma_{a}{}^{c}{}_{b} \nabla^+_{c} W^{b}\right) -\alpha'\Gamma_{a}{}^{c}{}_{b} \delbar\left(\nabla^+_{c} W^{b}\right)\notag \\
    &=T_{a b} W^{b} -\alpha'\Gamma_{a}{}^{c}{}_{b} {{R_d}^b}_cW^{d}+\alpha'\delbar\left(\Gamma_{a}{}^{c}{}_{b} \nabla^+_{c} W^{b}\right)-\alpha'\Gamma_{a}{}^{c}{}_{b}\nabla^+_{c} \delbar W^{b}\notag\\
    &=\delbar \tilde{\tau}_{a b} W^{b} +\alpha'\left[\delbar, \Gamma_{a}{}^{c}{}_{b} \nabla^+_{c}\right] W^{b}\notag\\
    &=\left[\delbar, \tilde{\tau}_{a b}+\alpha'\Gamma_{a}{}^{c}{}_{b} \nabla^+_{c}\right] W^{b}\:,
    \label{eq:LocTrivAnomaly} 
\end{align}
    where we have used Corollary \ref{cor:commute} in the third equality, and the fact that, by \eqref{eq:delbar.Gamma}, 
$$\tr( \Gamma_a\delbar \Gamma_d)= \Gamma_{a}{}^{c}{}_{b}\delbar \Gamma_{d}{}^{b}{}_{c}=\Gamma_{a}{}^c{}_b R_{d}{}^b{}_c$$ 
    in the fourth equality. Thus 
$$\bar{D}|_U
=\underbrace{\left[\begin{array}{cc}
    1 & -\tilde\tau-\alpha'\Gamma \cdot \nabla^+ \\
    0 & 1
\end{array}\right]}_{\phi^{-1}} {\circ\; \delbar \;\circ} \underbrace{\left[\begin{array}{cc}
    1 & \tilde\tau+\alpha'\Gamma \cdot \nabla^+ \\
    0 & 1
\end{array}\right]}_{\phi}
$$
    as claimed.
\end{proof}

\begin{remark}
    It should be noted that a similar operator and local trivialisation can be defined for any complex manifold with a Hermitian metric $(X,g)$, satisfying \eqref{eq: delbar Tab} for some $(2,1)$-form $T$.
\end{remark}

\subsection{Including the gauge sector}
\label{sec:CheckGauge}
We now wish to turn on the gauge bundle, and consider the full operator $\bar D$ in \eqref{eq: Dbar}. Recall $T$, $R$ and $F$ given in Notation \ref{notation1} and that locally we can write $F=\delbar A$ and $R=\delbar\Gamma$ where here, and throughout the section, we simply write $\delbar$ for the Dolbeault operator on $E$.

\begin{prop}\label{prop:local.barD.delbar}
   Given any sufficiently small  chart $U$ on $X$,  there exists a local section $\phi\in\Gamma(U,\End(Q))$, determined by $T$, $R$, $\Gamma$, $F$ and $A$, such that $\bar{D}$ in \eqref{eq: Dbar} satisfies
\begin{equation}
\label{eq:trivial}
    \bar{D}|_U=\phi^{-1}\circ \delbar\circ\phi.
\end{equation}
\end{prop}
\begin{proof}
    Recall that \eqref{eq:F2} can be written
\begin{equation}
\label{eq:GenBI}
    \delbar T=\tfrac{\alpha'}{2}\left(\tr \, R \wedge R - \tr F \wedge F\right)\:,
\end{equation}
where $T$, $R$ and $F$ are given in Notation \ref{notation1}. Locally, we can write $F=\delbar A$,  $R=\delbar\Gamma$ as we noted above and thus we can write \eqref{eq:GenBI} as
\begin{equation}
    \delbar\left( T_{ab}-\alpha'\tr(\Gamma_aR_b)+\alpha'\tr(A_aF_b)\right)=0\:.
\end{equation}
As in the proof of Proposition \ref{prop: D bar (loc)= delbar} leading up to equation \eqref{eq:DefTildeTau}, this implies that
\begin{equation}
\label{eq:Defb}
    T_{ab}-\alpha'\tr(\Gamma_aR_b)+\alpha'\tr(A_aF_b)=\delbar\tilde\tau_{ab}\:,
\end{equation}
for some locally defined $\tilde\tau$ which is a section of $(T^*X)^{1,0}\otimes (T^*X)^{1,0}$. 

If we then define $\phi$ by
\begin{equation}
\label{eq: phi matrix 3x3}
    \phi=\begin{bmatrix}
        1 & \alpha'A & \tilde\tau+\alpha'\Gamma\cdot\nabla^+  \\
        0 & 1 & A \\
        0 & 0 & 1
    \end{bmatrix} \quad 
\begin{gathered}
   (T^{1,0}X)^* \\
   \End(E)\\
   T^{1,0}X 
\end{gathered}\:,
\end{equation}
we see that it is invertible, and a similar computation leading to equation \eqref{eq:LocTrivAnomaly} shows that its inverse can be written as
\begin{equation}
\label{eq: phi^-1 matrix 3x3}
    \phi^{-1}=\begin{bmatrix}
        1 & -\alpha'A & \alpha'A\cdot A-(\tilde\tau+ \alpha'\Gamma\cdot\nabla^+)  \\
        0 & 1 & -A \\
        0 & 0 & 1
    \end{bmatrix},
\end{equation}
where
\begin{equation}\label{eq:AdotA}
  \left[(A\cdot A) W\right]_a=\tr(A_{a}A_d)W^d.
\end{equation}
We can then compute
\begin{equation}
    \phi^{-1}\circ\delbar\circ \phi=
    \begin{bmatrix}
        \delbar & \alpha'\sF & (\delbar\tilde\tau-\alpha'A\cdot F+\alpha'\Gamma\cdot R) +\alpha'\cR\nabla^+ \\
        0 & \delbar & \sF \\
        0 & 0 & \delbar
    \end{bmatrix},
\end{equation}
whre $A\cdot F$ and $\Gamma\cdot R$ are defined in a similar manner to \eqref{eq:AdotA}. 
In the top right corner, note that the action of $\delbar \tilde{\tau}-\alpha'A\cdot F +\alpha'\Gamma\cdot R$ is precisely the action of $\sT$, due to equation \eqref{eq:Defb}. Therefore $\bar{D}$ is locally trivialisable as desired in \eqref{eq:trivial}.
\end{proof}

\begin{remark}
    We see that the discussion above is valid as soon as we have \eqref{eq:GenBI} for \emph{some} $T$ (not necessarily defined by the torsion of the $\SU(3)$-structure), as we long as we suitably modify the definition of $\bar{D}$.  Note that we do not require the full anomaly cancellation \eqref{eq:F2}; in particular, we do not actually require the first Pontrjagin classes of $TX$ and $E$ to match (the sometimes called \emph{omalous} condition).
\end{remark}
 
%%%%%%%%%%%
\subsection{\v{C}ech cohomology}
\label{sec: rel to Cech coh}
We now wish to define our notion of \v{C}ech 
 cohomology. 
We start with the basic set up, which we will use throughout this subsection.
\begin{definition}
\label{dfn:Cech.setup}
      Let $\mathcal{U}=\{U_i\}$ be an atlas cover of  $X$  and suppose that each $\phi_i\in \Gamma(U_i,\End(Q))$ satisfies
\begin{equation}
\label{eq:phi.i}
    \bar{D}|_{U_i} =\phi_i^{-1}\circ\delbar\circ\phi_i
    \quad\text{for all $i$.}
\end{equation}
    Such a collection of pairs $\mathcal{P}=\{(U_i,\phi_i)\}$ is guaranteed to exist if $X$ is compact by Proposition \ref{prop:local.barD.delbar}. 

    We then let 
\begin{equation}
\label{eq:intersections}
    U_{i_1\ldots i_k}=\cap_{j=1}^kU_{i_j}
\end{equation}
    and let 
\begin{equation}
\label{eq:psi.ij}
    \psi_{ij} =\phi_i\circ\phi_j^{-1}\in \Gamma(U_{ij},\End(Q)).
\end{equation}
\end{definition}

Before proceeding with our definition of \v{C}ech cohomology, we make the following observations.
\begin{lemma}
\label{Lem: psi_ij}  
    The $\{\psi_{ij}\}$ in \eqref{eq:psi.ij}   satisfy
    \begin{enumerate}
        \item $\left[\delbar, \psi_{i j}\right]=0$, i.e. each $\psi_{ij}$ is holomorphic;
        \item the cocycle conditions:
$$ \psi_{ij}\circ\psi_{ji}=\id\quad \text{on}\; U_{ij}\quad\text{and}\quad
\psi_{ij}\circ \psi_{jk}\circ \psi_{k i}=\id
\quad\text{on}\;
U_{ijk}.
$$
    \end{enumerate}
\end{lemma}
\begin{proof}
    Property {\it (2)} follows directly from \eqref{eq:psi.ij}. For Property {\it (1)} we note that on the overlap $U_{ij}$ the $\bar{D}$ operator may be written in two ways
    \begin{equation}\label{eq:barD.Uij}
\phi_i^{-1}\circ\delbar\circ\phi_i=\bar{D}|_{U_{ij}}=\phi_j^{-1}\circ\delbar\circ\phi_j\:.
    \end{equation}
    Applying $\phi_i$ to the left hand side, and $\phi_j^{-1}$ to the right hand side of \eqref{eq:barD.Uij}, we get
    \begin{equation*}
        \delbar\circ\phi_i\circ\phi_j^{-1}-\phi_i\circ\phi_j^{-1}\circ\delbar=0\:,
    \end{equation*}
    which is the desired result.
\end{proof}

\begin{remark} Lemma \ref{Lem: psi_ij} suggests that we should think of the $\psi_{ij}$ in \eqref{eq:psi.ij} as `transition operators' defining a ‘holomorphic bundle’ structure on $Q$.
\end{remark}

We may now define the cochains and differential which we lead to our cohomology groups. 

\begin{definition}
\label{dfn:Cech.cochain}
    We define the $p$-cochains $C^p(Q,\mathcal{P})$ as sets  of pairs $\eta=\{(\eta_{i_0\ldots i_p},U_{i_0\ldots i_p})\}$, where we consider all intersections $U_{i_0\ldots i_p}$ of $p+1$ elements of $\mathcal{U}$ and $\eta_{i_0\ldots i_p}\in\Gamma(U_{i_0\ldots i_p},Q)$ satisfy 
\begin{align}
\label{eq:Cech.holo}
    \delbar\eta_{i_0\ldots i_p}&=0\:,\\
\label{eq:Cech.sym.1}
    \eta_{i_0\dots i_l\dots i_k\dots i_p}&=-\eta_{i_0\dots i_k\dots i_l\dots i_p}\:,\\
\label{eq:Cech.sym.2}
    \eta_{i_0\dots i_l\dots i_p}&=-\psi_{i_0i_l}\eta_{i_l\dots i_0\dots i_p}\:,
\end{align}
where there is no summation over repeated indices in \eqref{eq:Cech.sym.2} and $k>l>0$.  

Given  $\eta=\{(\eta_{i_0\ldots i_p},U_{i_0\ldots i_p})\}\in C^p(Q,\mathcal{P})$ we define $\delta\eta=\{((\delta\eta)_{i_0\ldots i_{p+1}},U_{i_0\ldots i_{p+1}})\}$ by
 \begin{equation}\label{eq:Cech.diff}
(\delta\eta)_{i_0\dots i_{p+1}}=\psi_{i_0 i_1} \eta_{i_1...i_{p+1}} + \sum_{k=1}^{p+1} (-1)^k \eta_{i_0\dots\hat{i}_k\dots i_{p+1}}
    \quad \text{on} \quad 
    U_{i_0,...,i_{p+1}}\:,
\end{equation}
where $\hat{j}$ denotes omission of the index $j$ and again there is no summation over repeated indices. We say that $\eta\in C^p(Q,\cP)$ is a $p$-cocycle if $\delta\eta=0$.
\end{definition}

\begin{remark}
In our notation in Definition \ref{dfn:Cech.cochain} we use $\mathcal{P}$ to stress that the cochains and maps $\delta$ depend on the $\phi_i$ (through the $\psi_{ij}$) and not just the open cover $\mathcal{U}$.
\end{remark}

In view of the properties established in Lemma \ref{Lem: psi_ij} for the cocycles $\{\psi_{ij}\}$, we now show that $\delta$ in \eqref{eq:Cech.diff} is indeed a differential, by a standard argument.

\begin{prop}\label{prop:Cech.diff}
We have that $\delta$ in Definition \ref{dfn:Cech.cochain} satisfies $\delta:C^p(Q,\cP)\to C^{p+1}(Q,\cP)$ and $\delta^2=0$.
\end{prop}

\begin{proof}
We first observe that 
\begin{equation}
    \delbar (\delta\eta)_{i_0\ldots i_{p+1}}=0 
\end{equation}
by Lemma \ref{Lem: psi_ij} and \eqref{eq:Cech.holo}.  We now wish to show that $\delta$ preserves the relations \eqref{eq:Cech.sym.1}--\eqref{eq:Cech.sym.2}. We proceed by induction. 

Consider a $0$-cochain $\eta=\{(\eta_{i_0},U_{i_0})\}$. We then have, using Lemma \ref{Lem: psi_ij} and \eqref{eq:Cech.diff}, that 
    \begin{equation}
(\delta\eta)_{i_0i_1}=\psi_{i_0i_1}\eta_{i_1}-\eta_{i_0}=-\psi_{i_0i_1}(\psi_{i_1i_0}\eta_{i_0}-\eta_{i_1})=-\psi_{i_0i_1}(\delta\eta)_{i_1i_0}\:.
        \label{eq:Ind0a}
    \end{equation}
    Next, consider a 1-cochain $\eta=\{(\eta_{i_0i_1},U_{i_0i_1})\}$. We have by \eqref{eq:Cech.sym.2} and \eqref{eq:Cech.diff} that
    \begin{align}
(\delta\eta)_{i_0i_1i_2}&=\psi_{i_0i_1}\eta_{i_1i_2}-\eta_{i_0i_2}+\eta_{i_0i_1}=-\psi_{i_0i_1}\psi_{i_1i_2}\eta_{i_2i_1}-\eta_{i_0i_2}+\eta_{i_0i_1}\notag\\
        &=-(\psi_{i_0i_2}\eta_{i_2i_1}-\eta_{i_0i_1}+\eta_{i_0i_2})=-(\delta\eta)_{i_0i_2i_1}\:,
        \label{eq:Ind0b}
    \end{align}
    and also
    \begin{align}
(\delta\eta)_{i_0i_1i_2}&=\psi_{i_0i_1}\eta_{i_1i_2}-\eta_{i_0i_2}+\eta_{i_0i_1}=-\psi_{i_0i_1}(-\eta_{i_1i_2}+\psi_{i_1i_0}\eta_{i_0i_2}-\psi_{i_1i_0}\eta_{i_0i_1})\notag\\
        &=-\psi_{i_0i_1}(\psi_{i_1i_0}\eta_{i_0i_2}-\eta_{i_1i_2}+\eta_{i_1i_0})=-\psi_{i_0i_1}(\delta\eta)_{i_1i_0i_2}\:.
        \label{eq:Ind0c}
    \end{align}
    Equations \eqref{eq:Ind0a}, \eqref{eq:Ind0b} and \eqref{eq:Ind0c} give the base case for the induction proof. 

    Next, consider a $p$-cochain $\eta=\{(\eta_{i_0i_1\dots i_p},U_{i_0i_1\dots i_p})\}$. We begin with showing \eqref{eq:Cech.sym.1} for $\delta\eta$. For $l,k\ge2$ it is straightforward to check using \eqref{eq:Cech.diff} and \eqref{eq:Cech.sym.1}--\eqref{eq:Cech.sym.2} for $\eta$ that interchanging indices $i_k$ and $i_l$ gives
    \begin{equation}\label{eq:Cech.skew.1}
        (\delta\eta)_{i_0\dots i_l\dots i_k\dots i_{p+1}}=-(\delta\eta)_{i_0\dots i_k\dots i_l\dots i_{p+1}}\:,
    \end{equation}
    while switching $i_1$ with $i_l$ for $l\ge2$ gives
    \begin{align}
        (\delta\eta)_{i_0i_1\dots i_l\dots i_{p+1}}&=\psi_{i_0 i_1} \eta_{i_1\dots i_l\dots i_{p+1}} + \sum_{k=1}^{p+1} (-1)^k \eta_{i_0\dots\hat{i}_k\dots i_{p+1}}\notag\\
        &=-\psi_{i_0 i_1}\psi_{i_1i_l} \eta_{i_l\dots i_1\dots i_{p+1}}
        -\eta_{i_0i_2\dots i_{p+1}}+(-1)^{l}\eta_{i_0i_1\dots\hat{i}_l\dots i_{p+1}}\nonumber\\
        &\quad+ \sum_{k=2,\: k\neq l}^{p+1} (-1)^k \eta_{i_0\dots\hat{i}_k\dots i_{p+1}}\notag\\
        &=-\psi_{i_0i_l} \eta_{i_l\dots i_1\dots i_{p+1}}-(-1)^l\eta_{i_0i_li_2\dots \hat{i}_l\dots i_{p+1}}+\eta_{i_0i_2\dots i_1\dots i_{p+1}}\nonumber\\
        &\quad-\sum_{k=2,\: k\neq l}^{p+1} (-1)^k \eta_{i_0i_l\dots i_1\dots\hat{i}_k\dots i_{p+1}}\notag\\
        %&=-(\psi_{i_0i_l} \gamma_{i_l\dots i_1\dots i_{p+1}}-\gamma_{i_0i_2\dots i_1\dots i_{p+1}}+\gamma_{i_0i_li_3\dots i_1\dots i_{p+1}}-\gamma_{i_0i_2i_4\dots i_1\dots i_{p+1}}+\dots)\notag\\
        &=-(\delta\eta)_{i_0i_l\dots i_1\dots i_{p+1}}\:,\label{eq:Cech.skew.2}
    \end{align}
    where we have also used Lemma \ref{Lem: psi_ij}.
Thus \eqref{eq:Cech.sym.1} holds for $\delta\eta$. Now we wish to show that $\delta\eta$ satisfies \eqref{eq:Cech.sym.2}. Note that as we have shown that $(\delta\eta)_{i_0i_l\dots i_{p+1}}$ is skew-symmetric in the  indices excluding $i_0$ in \eqref{eq:Cech.skew.1}--\eqref{eq:Cech.skew.2}, we only need to show \eqref{eq:Cech.sym.2} when interchanging $i_0$ and $i_1$. To that end, we compute
    \begin{align}
        (\delta\eta)_{i_0i_1\dots i_{p+1}}&=\psi_{i_0 i_1} \eta_{i_1\dots i_{p+1}} -\eta_{i_0i_2\dots i_{p+1}}+ \sum_{k=2}^{p+1} (-1)^k \eta_{i_0i_1\dots\hat{i}_k\dots i_{p+1}}\notag\\
        &=-\psi_{i_0i_1}\psi_{i_1i_0}\eta_{i_0i_2\dots i_{p+1}}+\psi_{i_0 i_1} \eta_{i_1\dots i_l\dots i_{p+1}}- \psi_{i_0i_1}\sum_{k=2}^{p+1} (-1)^k \eta_{i_1i_0\dots\hat{i}_k\dots i_{p+1}}\notag\\
        &=-\psi_{i_0i_1}\left(\psi_{i_1i_0}\eta_{i_0i_2\dots i_{p+1}}-\eta_{i_1\dots i_l\dots i_{p+1}}+\sum_{k=2}^{p+1} (-1)^k \eta_{i_1i_0\dots\hat{i}_k\dots i_{p+1}}\right)\notag\\
        &=-\psi_{i_0i_1}(\delta\eta)_{i_1i_0\dots i_{p+1}}\:,\label{eq:Cech.skew.3}
    \end{align}
    where we have again used Lemma \ref{Lem: psi_ij}, \eqref{eq:Cech.diff} and \eqref{eq:Cech.sym.1}--\eqref{eq:Cech.sym.2} for $\eta$.
    By induction, it follows from \eqref{eq:Cech.skew.1}--\eqref{eq:Cech.skew.3} that $\delta$ preserves the relations \eqref{eq:Cech.sym.1}--\eqref{eq:Cech.sym.2} as claimed.

  We deduce that $\delta$ defines a map between cochains as desired.  Finally, we show that $\delta^2=0$. Let $\eta=\{(\eta_{i_0i_1\dots i_p},U_{i_0i_1\dots i_p})\}$ be a $p$-cochain. We have, by \eqref{eq:Cech.diff}, 
    \begin{align}
        (\delta^2\eta)_{i_0i_1\dots i_{p+2}}&=\psi_{i_0i_1}\left(\psi_{i_1 i_2} \eta_{i_2\dots i_{p+2}} + \sum_{k=2}^{p+2} (-1)^{k+1} \eta_{i_1i_2\dots\hat{i}_k\dots i_{p+2}}\right)\label{eq:line.1}\\
        &-\psi_{i_0 i_2} \eta_{i_2\dots i_l\dots i_{p+2}} - \sum_{k=2}^{p+2} (-1)^{k} \eta_{i_0i_2\dots\hat{i}_k\dots i_{p+2}}\label{eq:line.2}\\
        &+\sum_{l=2}^{p+2}(-1)^l\psi_{i_0 i_1} \eta_{i_1\dots \hat{i}_l\dots i_{p+1}} + \sum_{k=1}^{p+2}\sum_{l=2}^{k-1}(-1)^l (-1)^{k} \eta_{i_0i_1\dots\hat{i}_l\dots\hat{i}_k\dots i_{p+2}}\label{eq:line.3}\\
        &+\sum_{k=1}^{p+2}\sum_{l=k+1}^{p+2}(-1)^{l+1} (-1)^k \eta_{i_0i_1\dots\hat{i}_k\dots\hat{i}_l\dots i_{p+2}}\:.\label{eq:line.4}
    \end{align}
    We see that the first term on the right-hand side  of \eqref{eq:line.1} cancels the first term in \eqref{eq:line.2} by Lemma \ref{Lem: psi_ij}. The second term on the right-hand side of \eqref{eq:line.1} cancels the first term in \eqref{eq:line.3} directly. The remaining terms from \eqref{eq:line.1}--\eqref{eq:line.4} may be written as:
    \begin{align}        (\delta^2\eta)_{i_0i_1\dots i_{p+2}}
        &= \sum_{k=1}^{p+2}\sum_{l=1}^{k-1}(-1)^l (-1)^{k} \eta_{i_0\dots\hat{i}_l\dots\hat{i}_k\dots i_{p+2}}\label{eq:line.1.a}\\
        &+\sum_{k=1}^{p+2}\sum_{l=k+1}^{p+2}(-1)^{l+1} (-1)^k \eta_{i_0\dots\hat{i}_k\dots\hat{i}_l\dots i_{p+2}}\:.\label{eq:line.2.a}
    \end{align}
  We see that the sum in \eqref{eq:line.1.a} is over all pairs $(k,l)$ with $1\leq l<k\leq p+2$, whereas the sum in \eqref{eq:line.2.a}, where the roles of $k,l$ in the indices, the sum is over pairs with $1\leq k<l\leq p+2$.  Therefore, swapping $k$ and $l$ in \eqref{eq:line.2.a}, we see that the sum there becomes
  \begin{equation}
\sum_{k=1}^{p+2}\sum_{l=1}^{k-1}(-1)^{k+1} (-1)^l \eta_{i_0\dots\hat{i}_l\dots\hat{i}_k\dots i_{p+2}},
  \end{equation}
  which precisely cancels with the sum in \eqref{eq:line.1.a}. It follows that $\delta^2=0$.
\end{proof}

Proposition \ref{prop:Cech.diff} immediately yields that the following cohomology groups are well-defined.

\begin{definition}
We define the \emph{\v{C}ech cohomology groups}  $\check{H}^p(Q)$ as the cohomology groups associated to the differential complex $(C^{\bullet}(Q,\cP),\delta)$ given by Definition \ref{dfn:Cech.cochain}.
\end{definition}

\subsection{Dolbeault theorem} Now that we have defined our \v{C}ech cohomology, we want to prove that the cohomology group $H^{0,1}_{\bar{D}}(Q)$, which governs the deformation theory of the heterotic $\SU(3)$ system, is isomorphic to $\check{H}^1(Q)$.  We will achieve this by constructing the isomorphism explicitly. The construction leading to the isomorphism follows a standard approach using a partition of unity. However, due to the transition maps being linear operators, rather than holomorphic linear maps, and as the functions of the partition of unity are non-(anti-)holomorphic, the argument needs to be modified in various places. We therefore give the argument leading to Theorem \ref{thm: Cech} in full.

Throughout this section we continue to assume we are in the setting of Definition \ref{dfn:Cech.setup}.  
We begin with the following useful lemma.

\begin{lemma}\label{lem:eta.i}
Refining the open cover $\mathcal{U}=\{U_i\}$ in Definition \ref{dfn:Cech.setup} if necessary, for all  $s\in \Gamma(Q^{0,1})$ with $\bar{D}s=0$ there exists $\eta_{i}\in \Gamma(U_i,Q)$ such that
\begin{equation}
\label{eq:Poincare}
    s|_{U_i}=\bar{D}\eta_i.
\end{equation}
For all such $s$ satisfying $s=\bar{D}t$, we choose $\eta_i=t|_{U_i}$.
\end{lemma}

\begin{proof}
 By \eqref{eq:phi.i} we know that $\bar{D}s=0$ on $U_i$ forces
 \begin{equation}\label{eq:delbar.s}
     \delbar(\phi_i s|_{U_i})=0.
 \end{equation}
Making $U_i$ smaller and modifying it so that the $\delbar$-Poincar\'e lemma applies, if necessary, we deduce from \eqref{eq:delbar.s} that there exists a section $\zeta_i$ of $Q|_{U_i}$ such that
\begin{equation}
\label{eq:zeta.i}
     \delbar\zeta_i=\phi_i s|_{U_i}.
 \end{equation}
 If we define
 \begin{equation}\label{eq:eta.i}
\eta_i=\phi_i^{-1}\zeta_i
 \end{equation}
 then \eqref{eq:zeta.i} and \eqref{eq:phi.i} imply \eqref{eq:eta.i}.  The final statement is obvious.
\end{proof}

Lemma \ref{lem:eta.i} enables us to prove the following key step in building our first map between cohomology groups.  We recall the notation in Definitions \ref{dfn:Cech.setup} and \ref{dfn:Cech.cochain}, including 
\eqref{eq:phi.i} and \eqref{eq:intersections}.

\begin{prop}
\label{prop:Phi} 
    Let $s\in \Gamma(Q^{0,1})$ with $\bar{D}s=0$ and let $\eta_i\in\Gamma(U_i,Q)$ be given by Lemma \ref{lem:eta.i}.  Define 
 \begin{equation}
     \Phi(s)=\{(\eta_{ij},U_{ij})\},
 \end{equation}
    where 
\begin{equation}
\label{eq:eta.ij}
     \eta_{ij}
     =\phi_i(\eta_i-\eta_j).
 \end{equation}
 Then $\Phi(s)\in C^1(Q,\cP)$ with $\delta\Phi(s)=0$; i.e.~$\Phi(s)$ is a 1-cocycle.
\end{prop}

\begin{proof}
To prove that $\Phi(s)$ is a 1-cochain we need to show that \eqref{eq:Cech.holo}--\eqref{eq:Cech.sym.2} hold for $\eta_{ij}$ in \eqref{eq:eta.ij}.   Note that condition \eqref{eq:Cech.sym.1} is vacuous and condition \eqref{eq:Cech.sym.2} can be stated as 
\begin{equation}\label{eq:eta.sym}
    \eta_{ij}=-\psi_{ij}\eta_{ji},
\end{equation}
which is immediate from the definition of $\eta_{ij}$ in \eqref{eq:eta.ij} and $\psi_{ij}$ in \eqref{eq:psi.ij}.

We are left with \eqref{eq:Cech.holo}. Since  \eqref{eq:eta.i} holds on $U_{ij}$ for both $i$ and $j$, we obtain a compatibility identity using \eqref{eq:phi.i}:
\[
  0=\bar{D}\eta_i-\bar{D}\eta_j=  \phi_{i}^{-1} \delbar\left(\phi_{i} \eta_{i}\right) 
    - \phi_{j}^{-1} \delbar\left(\phi_{j} \eta_{j}\right).
\]
Hence, by definition of $\psi_{ij}$ in \eqref{eq:psi.ij}, 
\[
\delbar\left(\phi_{i} \eta_{i}\right) 
    - \psi_{ij}\delbar\left(\phi_{j} \eta_{j}\right) = 0.
\]
By Lemma \ref{Lem: psi_ij}, $\psi_{ij}$ commutes with $\delbar$, so
\begin{equation}
\label{eq: gamma_ij is hol}
   { \delbar \eta_{ij}=\delbar\left(\phi_{i} \eta_{i}-\phi_{i} \eta_{j}\right)=0.}
\end{equation}
In other words, \eqref{eq:Cech.holo} holds.

Finally, it is immediate from \eqref{eq:eta.ij} that
\begin{equation}
\label{eq:cocycle}
    \phi_{i}^{-1} \eta_{ij} + \phi_{j}^{-1} \eta_{jk} + \phi_{k}^{-1} \eta_{ki} = 0
    \quad \text{on}\; 
    U_{ijk}\:.
\end{equation}
It follows from the definition of $\psi_{ij}$ and \eqref{eq:eta.sym}  that \eqref{eq:cocycle} is equivalent to
\begin{equation}
    \label{eq:cocycle2}
    \psi_{ij} \eta_{jk} - \eta_{ik} + \eta_{ij}= 0
    \quad \text{on}\; 
    U_{ijk}\:.
\end{equation}
Equation \eqref{eq:cocycle2} is the statement that $\delta\Phi(s)=0$ by \eqref{eq:Cech.diff}.
\end{proof}

\begin{corollary}\label{cor:Phi}
The map $\Phi$ in Proposition \ref{prop:Phi} induces a well-defined linear map $\Phi:H^{0,1}_{\bar{D}}(Q)\to \check{H}^1(Q)$.
\end{corollary}

\begin{proof}
Given Proposition \ref{prop:Phi}, for well-definedness one need only check that $[\Phi(s)]=0\in \check{H}^1(Q)$, when $[s]=0\in H^{0,1}_{\bar{D}}(Q)$.  Indeed, if any representative of the latter has the form $s=\bar{D}t$, then $\eta_{ij}=0$ by \eqref{eq:eta.ij} and Lemma \ref{lem:eta.i}, so $[\Phi(s)]=0$ trivially.  Linearity of $\Phi$ is clear by construction.
\end{proof}

We now wish to define a linear map from $\check{H}^1(Q)$ to $H^{0,1}_{\bar{D}}(Q)$, which we want to be the inverse of $\Phi$ in Corollary \ref{cor:Phi}.  Throughout this construction, we recall that given the open cover $\mathcal{U}=\{U_i\}$, by re-indexing and removing unnecessary open sets, we may assume that it is countable and indexed by $i\in\mathbb{N}$, and that there exists a  partition of unity $\{\rho_i\}$ such that
\begin{equation}\label{eq:po1}
    \supp\rho_i\subseteq U_i.
\end{equation}

Our goal is the following result.

\begin{prop}\label{prop:Psi}
Let $\eta\in C^{1}(Q,\cP)$ with $\delta\eta=0$ and recall $\Phi$ from Proposition \ref{prop:Phi}. There exists $\Psi(\eta)\in \Gamma(Q^{0,1})$ with $\bar{D}\Psi(\eta)=0$ such that the following holds.
\begin{itemize}
    \item[\emph{(1)}] If $\eta=\Phi(s)$ then $[\Psi\circ\Phi(s)]=[\Psi(\eta)]=[s]\in H^{0,1}_{\bar{D}}(Q)$.
    \item[\emph{(2)}]  $[\Phi\circ\Psi(\eta)]=[\eta]\in \check{H}^1(Q,\cP)$.
 \end{itemize}
\end{prop}

The proof of Proposition \ref{prop:Psi} is rather lengthy, so we break it into several steps.  We begin by constructing the map $\Psi$.

Let $\eta=\{(\eta_{ij},U_{ij}\}$ be a 1-cocycle. 
Throughout this construction we will not sum over repeated indices $i,j$, but will sum over repeated indices $a,b,c$, and we will omit antiholomorphic indices.

We recall that
\begin{equation}
\label{eq:phi.i.matrix}
    \phi=\begin{bmatrix}
        1 & \alpha'A_i & \tilde\tau_i+\alpha'\Gamma_i\cdot\nabla^+  \\
        0 & 1 & A_i \\
        0 & 0 & 1
    \end{bmatrix} \quad 
\begin{gathered}
   (T^{1,0}X)^* \\
   \End(E)\\
   T^{1,0}X 
\end{gathered}\:,
\end{equation}
where $\tilde\tau_i$ is a section of $(T^{1,0}X)^*\otimes (T^{1,0}X)^*$ over $U_i$, 
\begin{equation}\label{eq:A.Gamma.i}
F|_{U_i}=\delbar A_i\qandq R|_{U_i}=\delbar\Gamma_i.
\end{equation}
We also recall that
\begin{equation}
\label{eq:phi.i.inverse.matrix}
    \phi^{-1}=\begin{bmatrix}
        1 & -\alpha'A_i & \alpha'A_i\cdot A_i-(\tilde\tau_i+ \alpha'\Gamma_i\cdot\nabla^+)  \\
        0 & 1 & -A_i \\
        0 & 0 & 1
    \end{bmatrix} \quad 
\begin{gathered}
   (T^{1,0}X)^* \\
   \End(E)\\
   T^{1,0}X 
\end{gathered}\:.
\end{equation}
We may also decompose $\eta_{ij}$ as 
\begin{equation}\label{eq:eta.ij.vector}
    \eta_{ij}=\begin{bmatrix}
        \kappa_{ij}\\ \gamma_{ij}\\ W_{ij}
    \end{bmatrix}\quad 
\begin{gathered}
   (T^{1,0}X)^* \\ 
   \End(E) \\
   T^{1,0}X 
\end{gathered}.
\end{equation}

We want to construct 
\begin{equation}\label{eq:Psi.eta.vector}
\Psi(\eta)=\begin{bmatrix}
    \kappa\\ \gamma \\ W
\end{bmatrix} \in\Gamma(Q^{0,1})
\end{equation}
such that $\bar{D}\Psi(\eta)=0$, i.e.
\begin{equation}\label{eq:Dbar.Psi}
    \delbar \kappa+\alpha'\cF\gamma +(\sT+\alpha'\cR\nabla^+)W=0,\quad \delbar\gamma+\cF W=0 \qandq \delbar W=0.
\end{equation} We recall that \eqref{eq:Dbar.Psi} is equivalent to 
\begin{equation}\label{eq:Dbar.Psi.2}
(\phi_i^{-1}\circ\delbar\circ\phi_i)\Psi(\eta)|_{U_i}=0
\end{equation}
begin satisfied for all $i$ by Proposition \ref{prop:local.barD.delbar}.

\begin{remark} Note that the $T^{1,0}X$-component  of \eqref{eq:Dbar.Psi.2} is the final equation in \eqref{eq:Dbar.Psi}, so the construction for this component $W$ of $\Psi(\eta)$ in \eqref{eq:Psi.eta.vector} will follow the standard approach. Further, the $\End(E)$-component of \eqref{eq:Dbar.Psi.2} is the second equation of \eqref{eq:Dbar.Psi}, which differs from the usual $\delbar\gamma=0$ by a linear term in $W$, so again $\gamma$ in $\Psi(\eta)$ can be constructed by standard methods.  However, we have to work harder to construct the $(T^{1,0}X)^*$-component $\kappa$ in \eqref{eq:Psi.eta.vector}.
\end{remark}

We first define on $U_i=\cup_jU_{ij}$ the following local section $\tilde{s}_i$ of $Q^{0,1}$:
\begin{equation}\label{eq:tilde.s.i}
  \tilde{s}_i =\sum_j\phi_i^{-1}\delbar(\eta_{ij}\rho^j)=\sum_j\phi_i^{-1}(\eta_{ij}\delbar\rho^j),
\end{equation}
using that $\delbar\eta_{ij}=0$.  We now observe the following.

\begin{lemma}
\label{lem:tilde.s.i} 
    Suppose that $\eta=\{(\eta_{ij},U_{ij})\}$ is a 1-cocycle. Then the $\tilde{s}_i \in \Gamma(U_i,Q^{0,1})$ given in \eqref{eq:tilde.s.i} satisfy $\bar{D}\tilde{s}_i=0$ and, on $U_{ik}$, we have
\begin{align}
\label{eq:tilde.s.i.k}
    \tilde{s}_i - \tilde{s}_k 
    &= \alpha' \sum_j \left( \Gamma_k \cdot (W_{kj} \del \delbar \rho^j) 
    - \Gamma_i \cdot (W_{ij} \del \delbar \rho^j) \right) \otimes \be_1,
\end{align}
where
\begin{equation}
    \Gamma_i \cdot (W_{ij} \del \delbar \rho^j) 
    = (\Gamma_i)_a{}^b{}_c dz^a \otimes (W_{ij})^c \del_b \delbar \rho^j.
\end{equation}
\end{lemma}

\begin{proof}
    Clearly, by \eqref{eq:phi.i}, we have $\bar{D}\tilde{s}_i=0$. Moreover, from \eqref{eq:phi.i.inverse.matrix} and \eqref{eq:tilde.s.i}, we see that
    \begin{equation}\label{eq:tilde.s.i.2}
        \tilde{s}_i = \sum_j (\phi_i^{-1} \eta_{ij}) \delbar \rho^j 
        - \alpha' \sum_j \left( \Gamma_i \cdot (W_{ij} \del \delbar \rho^j) \otimes \be_1 \right).
    \end{equation}
    
    We now want to compare $\tilde{s}_i$ and $\tilde{s}_k$ on $U_{ik}$. From the condition \eqref{eq:eta.sym}, we know that on $U_{jk}$:
    \begin{equation}\label{eq:eta.sym.jk}
        \phi_j^{-1} \eta_{jk} = - \phi_k^{-1} \eta_{kj}.
    \end{equation}
    We deduce that on $U_{ijk}$:
    \begin{equation}\label{eq:eta.sym.ijk}
        \phi_i^{-1} \eta_{ij} - \phi_k^{-1} \eta_{kj} = \phi_i^{-1} \eta_{ij} + \phi_j^{-1} \eta_{jk} = -\phi_k^{-1} \eta_{ki},
    \end{equation} 
    from \eqref{eq:cocycle} and \eqref{eq:eta.sym.jk}. Hence, on $U_{ik} = \cup_j U_{ijk}$, we have:
    \begin{equation}\label{eq:tilde.s.vanish}
        \sum_j (\phi_i^{-1} \eta_{ij} - \phi_k^{-1} \eta_{kj}) \delbar \rho^j = -\sum_j \phi_k^{-1} \eta_{ki} \delbar \rho^j = -\phi_k^{-1} \eta_{ki} \delbar \sum_j \rho^j = 0,
    \end{equation}
    as $\{\rho^j\}$ is a partition of unity.

    Therefore, using \eqref{eq:tilde.s.i.2} and \eqref{eq:tilde.s.vanish}, we find that on $U_{ik}$:
    \begin{align}
        \tilde{s}_i - \tilde{s}_k 
        &= -\sum_j \phi_k^{-1} \eta_{ki} \delbar \rho^j 
        + \alpha' \sum_j \left( \Gamma_k \cdot (W_{kj} \del \delbar \rho^j) 
        - \Gamma_i \cdot (W_{ij} \del \delbar \rho^j) \right) \otimes \be_1 \nonumber \\
        &= \alpha' \sum_j \left( \Gamma_k \cdot (W_{kj} \del \delbar \rho^j) 
        - \Gamma_i \cdot (W_{ij} \del \delbar \rho^j) \right) \otimes \be_1,
    \end{align}
    as desired.
\end{proof}
We see from \eqref{eq:tilde.s.i.k} that $\tilde{s}_i$ may not patch together to give a globally defined section of $Q^{0,1}$.  We therefore need to modify $\tilde{s}_i$.  This is achieved by the following lemma.

\begin{lemma}\label{lem:zeta.ij}
   Suppose that $\eta=\{(\eta_{ij},U_{ij})\}$ is a 1-cocycle as in \eqref{eq:eta.ij.vector}.  Then there exist $\beta_i\in\Omega^{0,1}(T^{1,0}X)$ over $U_i$ such that the $T^{1,0}X$-valued $(0,1)$-forms $\zeta_{ij}$ on $U_{ij}$  given by
\begin{equation}\label{eq:zeta.ij}
    (\zeta_{ij})^b{}_a=(\Gamma_i)_{a}{}^b{}_c(W_{ij})^c-(\beta_i)^b{}_a+(\beta_j)^b{}_a
\end{equation}
satisfy
\begin{equation}\label{eq:zeta.ij.holo}
    \delbar\zeta_{ij}=0.
\end{equation}
\end{lemma}

\begin{proof}
    We first see from the cocycle condition \eqref{eq:cocycle} and our definition of $1$-cochains, together 
with \eqref{eq:phi.i.inverse.matrix} and \eqref{eq:eta.ij.vector}%\eqref{eq:phi.i.inverse.simple} and \eqref{eq:eta.ij.simple}
, that $\{(W_{ij},U_{ij})\}$ is a standard 1-cocycle in usual \v{C}ech cohomology, i.e.
\begin{equation}\label{eq:W.cocycle}
    W_{ij}+W_{jk}+W_{ki}=0\quad \text{on $U_{ijk}$}.
\end{equation}
By the usual $\delbar$-Dolbeault theorem, $\{W_{ij},U_{ij}\}$ defines $\tilde{W}\in \Omega^{0,1}(T^{1,0}X)$ with $\delbar \tilde{W}=0$.  Moreover, again by the $\delbar$-Dolbeault theorem, for each $i$ there exist sections $V_i,Y_i$ of $T^{1,0}X$ over $U_i$ such that 
\begin{equation}\label{eq:V.i}
\tilde{W}|_{U_i}=\delbar V_i,\quad \delbar Y_i=0
\end{equation}
and
\begin{equation}\label{eq:W.ij}
W_{ij}=V_i-V_j+Y_i-Y_j.
\end{equation}

We note from \eqref{R-conv} and $\delbar \tilde{W}=0$ that, on $U_i$, we have
\begin{equation}\label{eq:beta.i.a}
    \delbar(\nabla_a\tilde{W}^b)=[\delbar,\nabla_a]\tilde{W}^b=R_{a}{}^b{}_cW^c=R_a{}^b{}_c\delbar(V_i)^c,
\end{equation}
where we used \eqref{eq:V.i} in the last step.  Therefore, as $\delbar R=0$, we may use the $\delbar$-Poincar\'e Lemma on $U_i$ to deduce from \eqref{eq:beta.i.a} the existence of a $T^{1,0}X$-valued $(0,1)$-form $\tilde{\beta}_i$ on $U_i$ such that
\begin{equation}\label{eq:tilde.beta.i}
\delbar(\tilde{\beta}_i)^b{}_{a}=R_{a}{}^b{}_c(V_i)^c-\nabla_a\tilde{W}^b.
\end{equation}
Recalling \eqref{eq:A.Gamma.i} %\eqref{eq:Gamma.i} 
and the $Y_i$ in \eqref{eq:V.i}--\eqref{eq:W.ij}, we then define $\beta_i$ on 
\begin{equation}\label{eq:beta.i}
(\beta_i)^b{}_a=(\tilde{\beta}_i)^b{}_a+(\Gamma_i)_{a}{}^b{}_c(Y_i)^c.
\end{equation}
We deduce from \eqref{eq:A.Gamma.i}, \eqref{eq:V.i} and \eqref{eq:tilde.beta.i} that
\begin{equation}\label{eq:delbar.beta.i}
    \delbar(\beta_i)^{b}{}_a=R_{a}{}^b{}_c(V_i)^c-\nabla_a\tilde{W}^b+R_a{}^b{}_c(Y_i)^c.
\end{equation}

The upshot is that on $U_{ij}$ we have
\begin{equation}\label{eq:delbar.beta.i.j}
    \delbar(\beta_i-\beta_j)^{b}{}_a=R_{a}{}^b{}_c(W_{ij})^c=\delbar\left((\Gamma_i)_{a}{}^b{}_c(W_{ij})^c\right),
\end{equation}
using %\eqref{eq:Gamma.i}
\eqref{eq:A.Gamma.i}, $\delbar W_{ij}=0$ and \eqref{eq:delbar.beta.i}.   This gives the result.
\end{proof}

Given Lemma \ref{lem:zeta.ij}, we therefore define on $U_i$ a new section $s_i$ of $Q^{0,1}$:
\begin{equation}\label{eq:s.i}
    s_i = \tilde{s}_i + \alpha' \sum_j \left( \zeta_{ij} \cdot \del \delbar \rho^j \right) \otimes \be_1,
\end{equation}
where
\begin{equation}\label{eq:zeta.cdot}
    \zeta_{ij} \cdot \del \delbar \rho^j = (\zeta_{ij})^b{}_a dz^a \otimes \del_b \delbar \rho^j.
\end{equation}

We now conclude this discussion by constructing the map $\Psi$. %in this setting.

\begin{lemma}\label{lem:Psi}%\label{lem:Psi.simple}
   Suppose that %$E$ has trivial structure group and
   $\eta$ is a 1-cocycle. Then there exists $s=\Psi(\eta)\in\Gamma(Q^{0,1})$ with $\bar{D}s=0$ such that $$s|_{U_i}=s_i$$
    given in \eqref{eq:s.i}.
\end{lemma}

\begin{proof}
    From Lemma \ref{lem:tilde.s.i}, \eqref{eq: Dbar}, \eqref{eq:zeta.ij.holo} and \eqref{eq:s.i}, we quickly see that $\bar{D}s_i=0$, so we just need to show that the $s_i$ agree on all overlaps $U_{ik}$.

 We compute from \eqref{eq:tilde.s.i.k} and \eqref{eq:s.i} that the $(T^{1,0}X)^*$-component of $s_i-s_k$ is
 \begin{equation}\label{eq:s.i.k}
     \alpha'\sum_j \Gamma_k\cdot (W_{kj}\del\delbar\rho^j)-\Gamma_i\cdot(W_{ij}\del\delbar\rho^j)+\zeta_{ij}\cdot\del\delbar\rho^j-\zeta_{kj}\cdot\del\delbar\rho^j.
 \end{equation}
 It follows directly by substituting \eqref{eq:zeta.ij} in \eqref{eq:s.i.k} that it becomes
 \begin{equation}
     \alpha'\sum_j(-\beta_i\cdot\del\delbar\rho^j+\beta_j\cdot\del\delbar\rho^j+\beta_k\cdot\del\delbar\rho^j-\beta_j\cdot\del\delbar\rho^j)=\alpha'(\beta_k-\beta_i)\cdot\sum_j\del\delbar\rho^j=0,
 \end{equation}
 where we use a similar notation to \eqref{eq:zeta.cdot} and the fact that $\sum_j\rho^j=1$.  
 
 Since the $T^{1,0}X$-component and the $\End(E)$-component of $s_i-s_k$ are both zero by \eqref{eq:tilde.s.i.k} and \eqref{eq:s.i}, the result now follows.
\end{proof}

It follows from Lemma \ref{lem:Psi}  that any 1-cocycle $\eta$ defines an element $[\Psi(\eta)]\in H^{0,1}_{\bar{D}}(Q)$.  In particular, this is true if $\eta=\Phi(s)$ for $s\in\ker\bar{D}\subseteq\Gamma(Q^{0,1})$. 

We now demonstrate that property (1) in Proposition \ref{prop:Psi} holds.
\begin{lemma}\label{lem:Psi.Phi}
Let $s\in \Gamma(Q^{0,1})$ with $\bar{D}s=0$.  Then for $\Psi$ as in Lemma \ref{lem:Psi} we have that $[\Psi\circ\Phi(s)]=[s]\in H^{0,1}_{\bar{D}}(Q)$. 
\end{lemma}

\begin{proof} 
    For each $i$, let $\eta_i$ be given by Lemma \ref{lem:eta.i} and let $\eta_{ij}$ be as in \eqref{eq:eta.ij}. By Proposition \ref{prop:Phi} and Lemma \ref{lem:Psi}, we see that
    \begin{equation}
    \label{eq:Psi.Phi.U.i}
        \Psi\circ\Phi(s)|_{U_i} = \sum_j \phi_i^{-1} \delbar(\phi_i(\eta_i - \eta_j)\rho^j) + \alpha' \sum_j \left( \zeta_{ij} \cdot \del \delbar \rho^j \right) \otimes \be_1,
    \end{equation}
    where $\zeta_{ij}$ is given as in Lemma \ref{lem:zeta.ij}, and
    \begin{equation}
    \label{eq:s.U.i}
        s|_{U_i} = \phi_i^{-1} \delbar(\phi_i \eta_i).
    \end{equation}
    We may therefore calculate, using the same computation leading to \eqref{eq:tilde.s.i.2} and Lemma \ref{lem:zeta.ij}, that
    \begin{align}
        s|_{U_i} & - \Psi\circ\Phi(s)|_{U_i} \nonumber \\
        &= \phi_i^{-1} \delbar(\phi_i \eta_i) - \sum_j \phi_i^{-1} \delbar(\phi_i (\eta_i - \eta_j) \rho^j) 
        - \alpha' \sum_j \left( \zeta_{ij} \cdot \del \delbar \rho^j \right) \otimes \be_1 \nonumber \\
        &= \phi_i^{-1} \delbar(\phi_i \eta_i) - (\phi_i^{-1} \circ \delbar \circ \phi_i) \left( \sum_j (\eta_i - \eta_j) \rho^j \right) \nonumber \\
        & \quad + \alpha' \delbar \left( \sum_j (\Gamma_i \cdot W_{ij} - \zeta_{ij}) \cdot \del \rho^j \right) \otimes \be_1 \nonumber \\
        &= \bar{D} \left( \sum_j \eta_j \rho^j \right) + \alpha' \bar{D} \left( \sum_j (\beta_i - \beta_j) \cdot \del \rho^j \right) \otimes \be_1 \nonumber \\
        &= \bar{D} \left( \sum_j \eta_j \rho^j - \alpha' \left( \sum_j \beta_j \del \rho^j \right) \otimes \be_1 \right),\label{eq:globally.exact.1}
    \end{align}
    where we used $\delbar \zeta_{ij} = 0$, \eqref{eq: Dbar}, and that $\{\rho^j\}$ is a partition of unity.
    The term \eqref{eq:globally.exact.1} is independent of $i$ and globally exact, which gives the result.
\end{proof}

We finally show that property (2) in Proposition \ref{prop:Psi} holds, to complete the proof.

\begin{lemma}
    For a 1-cocycle $\eta$ and $\Psi$ as in Lemma \ref{lem:Psi}, we have that $[\Phi\circ\Psi(\eta)] = [\eta] \in \check{H}^1(Q)$. 
\end{lemma}

\begin{proof} 
    Let $\eta = \{(\eta_{ij}, U_{ij})\}$ be a 1-cocycle.  
    For ease of notation, we let
    \begin{equation}
    \label{eq:tilde_eta_def}
        \tilde{\eta} = \{(\tilde{\eta}_{ij}, U_{ij})\} = \Phi \circ \Psi(\eta).
    \end{equation} 
    Recall that we have $W_{ij}$ as in \eqref{eq:eta.ij.vector}, and that we have $\beta_i$ and $\zeta_{ij}$ as in Lemma \ref{lem:zeta.ij}.
    
    By Corollary \ref{cor:Phi} and the construction of $\Psi$ in Lemma \ref{lem:Psi}, we see that, up to a term which lies in the image of $\delta$ on $0$-cochains,
    \begin{equation}
    \label{eq:tilde.eta.ij}
        \tilde{\eta}_{ij} = \phi_i (\tilde{\eta}_i - \tilde{\eta_j}),
    \end{equation}
    where, on $U_i$, we define
    \begin{equation}
    \label{eq:tilde.eta.i}
        \tilde{\eta}_i = \phi_i^{-1} \sum_j \left( \eta_{ij} \rho^j + \alpha' \left( \zeta_{ij} \cdot \del \rho^j \right) \otimes \be_1 \right).
    \end{equation}
    
    Substituting \eqref{eq:tilde.eta.i} into \eqref{eq:tilde.eta.ij}, we see that on $U_{ik}$ we have
    \begin{equation}
    \label{eq:tilde.eta.i.k}
        \tilde{\eta}_{ik} = \sum_j \left( \eta_{ij} \rho^j - \psi_{ik} (\eta_{kj} \rho^j) \right) + \alpha' \sum_j \left( (\zeta_{ij} - \zeta_{kj}) \del \rho^j \right) \otimes \be_1,
    \end{equation}
    where we have used \eqref{eq:phi.i.matrix} and \eqref{eq:phi.i.inverse.matrix}.
    
    We now see from \eqref{eq:phi.i.matrix} and \eqref{eq:phi.i.inverse.matrix} that
    \begin{align}
        -\psi_{ik} (\eta_{kj} \rho^j) 
        &= -(\phi_{ik} \eta_{kj}) \rho^j - \alpha' \left( \Gamma_i \cdot W_{kj} \cdot \del \rho^j + \Gamma_k \cdot W_{kj} \cdot \del \rho^j \right) \otimes \be_1 \nonumber \\
        &= -(\psi_{ik} \eta_{kj}) \rho^j + \alpha' \left( (\beta_k - \beta_j + \zeta_{kj} - \Gamma_i \cdot W_{kj}) \cdot \del \rho^j \right) \otimes \be_1. \label{eq:third.term}
    \end{align}
    
    As we noted before in \eqref{eq:W.cocycle}, we have that 
    \begin{equation}
    \label{eq:W.cocycle.2}
        W_{kj} = W_{ij} + W_{ki},
    \end{equation}
    and since $\{\rho^j\}$ is a partition of unity, we have that
    \begin{equation}
    \label{eq:third.term.vanish}
        \sum_j \Gamma_i \cdot W_{ki} \cdot \del \rho^j = 0.
    \end{equation}
    
    Summing \eqref{eq:third.term} over $j$ and using \eqref{eq:W.cocycle.2} and \eqref{eq:third.term.vanish}, we see that
    \begin{align}
        -\sum_j \psi_{ik} (\eta_{kj} \rho^j) 
        &= -\sum_j (\psi_{ik} \eta_{kj}) \rho^j + \alpha' \sum_j \left( (\beta_k - \beta_j + \zeta_{kj} - \Gamma_i \cdot W_{ij}) \cdot \del \rho^j \right) \otimes \be_1 \nonumber \\
        &= -\sum_j (\psi_{ik} \eta_{kj}) \rho^j + \alpha' \sum_j \left( (\beta_k - \beta_j + \zeta_{kj} - \beta_i + \beta_j - \zeta_{ij}) \cdot \del \rho^j \right) \otimes \be_1 \nonumber \\
        &= -\sum_j (\psi_{ik} \eta_{kj}) \rho^j + \alpha' \sum_j \left( (\zeta_{kj} - \zeta_{ij}) \cdot \del \rho^j \right) \otimes \be_1, \label{eq:third.term.2}
    \end{align}
    where we again used that $\sum_j \rho^j = 1.$
    
    Substituting \eqref{eq:third.term.2} into \eqref{eq:tilde.eta.i.k}, we deduce from \eqref{eq:cocycle2} that
    \begin{align*}
        \tilde\eta_{ik} &= \sum_j (\eta_{ij} - \psi_{ik} \eta_{kj}) \rho^j = \sum_j \eta_{ik} \rho^j = \eta_{ik}
    \end{align*}
    as desired.
\end{proof}

Having completed the proof of Proposition \ref{prop:Psi}, we have the following corollary.

\begin{corollary}
\label{cor:Psi}
    The map $\Psi$ from Proposition \ref{prop:Psi} induces a well-defined linear map $\Psi:\check{H}^1(Q)\to H^{0,1}_{\bar{D}}(Q)$. 
\end{corollary}

\begin{proof}
    By Proposition \ref{prop:Psi} and its proof, we see that if the map $\Psi$ between cohomology groups is well-defined, it is linear. We therefore need only show that if $[\eta] = 0 \in \check{H}^1(Q)$, then $[\Psi(\eta)] = 0 \in H^{0,1}_{\bar{D}}(Q)$. 

    Suppose $\eta = \{(\eta_{ij}, U_{ij})\}$ lies in the image of $\delta$ acting on $0$-cochains. Then for all $i$, there exist holomorphic sections $\eta_i$ of $Q$ over $U_i$ such that
    \begin{equation}
    \label{eq:eta.ij.0}
        \eta_{ij} = \eta_i - \psi_{ij} \eta_j.
    \end{equation}
    In particular, this means that if we define $W_{ij}$ as in \eqref{eq:eta.ij.vector}, we see that
    \begin{equation}
    \label{eq:W.ij.0}
        W_{ij} = W_i - W_j,
    \end{equation}
    for holomorphic $W_i$ which are $T^{1,0}X$-components of $\eta_i$.  

    Therefore, if we let $s = \Psi(\eta)$, $s_i = s|_{U_i}$, and recall Lemmas \ref{lem:zeta.ij} and \ref{lem:Psi}, we see that
    \begin{align}
        s_i 
        &= \phi_i^{-1} \delbar \sum_j (\eta_i - \psi_{ij} \eta_j) \rho^j + \alpha' \phi_i^{-1} \delbar \sum_j \left( \zeta_{ij} \del \rho^j \right) \otimes \be_1 \nonumber \\
        &= -\phi_i^{-1} \delbar \sum_j \phi_i \left( (\phi_j^{-1} \eta_j) \rho^j \right) 
        + \alpha' \delbar \sum_j \left( (\Gamma_i \cdot W_j + \zeta_{ij}) \cdot \del \rho^j \right) \otimes \be_1, \label{eq:s.i.0}
    \end{align}
    where we once again used that $\sum_j \rho^j = 1$. 

    We now observe from Lemma \ref{lem:zeta.ij} and \eqref{eq:W.ij.0} that
    \begin{equation}
    \label{eq:Gamma.W.0}
        \Gamma_i \cdot W_j = \beta_j - \beta_i - \zeta_{ij} + \Gamma_i \cdot W_i.
    \end{equation}
    Using \eqref{eq:Gamma.W.0} in \eqref{eq:s.i.0}, we conclude that
    \begin{equation}
    \label{eq:globally.exact.0}
        s_i = \bar{D} \left( \sum_j (\phi_j^{-1} \gamma_j) \rho^j + \sum_j \left( \beta_j \del \rho^j \right) \otimes \be_1 \right).
    \end{equation}
    The right-hand side of \eqref{eq:globally.exact.0} is independent of $i$, and so $s = \Psi(\eta)$ is globally $\bar{D}$-exact as required.
\end{proof}

We now conclude this section by putting together all of our results to prove our Dolbeault-type theorem.
\begin{theorem}
\label{thm: Cech}
    The cohomology groups $H^{0,1}_{\bar D}(Q)$ and $\check{H}^1(Q)$ are isomorphic.  Moreover, the maps $\Phi:H^{0,1}_{\bar D}(Q)\to \check{H}^1(Q)$ and $\Psi:\check{H}^1(Q)\to H^{0,1}_{\bar{D}}(Q)$ given in Corollary \ref{cor:Phi} and \ref{cor:Psi} respectively are inverses of each other.
\end{theorem}
\begin{proof}
The result follows directly from Corollary \ref{cor:Phi}, Corollary \ref{cor:Psi} and Proposition \ref{prop:Psi}.
\end{proof}
\begin{remark} The isomorphism proved in Theorem \ref{thm: Cech} between the $\bar{D}$-cohomology and \v{C}ech cohomology groups  should also work similarly for higher degrees, but we leave a general proof of such a Dolbeault theorem to future work. 
\end{remark}

\appendix

\section{Heterotic moduli}
\label{app:Moduli}

The aim of this appendix is to briefly review the previous moduli theory for the heterotic $\SU(3)$ system as presented in \cite{McOrist2022}, and to augment that study to demonstrate further the links to the cohomology theory studied in this article.

We assume that $X$ is a complex 3-dimensional manifold endowed with a Hermitian $(1,1)$-form $\omega$ and a nowhere vanishing $(3,0)$-form $\Omega$.  We also assume that $\nabla$ is a connection on $TX$, which is the Chern connection associated to $\omega$, or the Chern connection associated to a Hermitian--Einstein metric on $TX$ (so that $\nabla$ is Hermitian Yang--Mills): this marks a difference from Notation \ref{notation1}.  We further assume that $A$ is a unitary connection on a holomorphic vector $E$ with a fibrewise Hermitian metric.  By the \emph{F-terms} we shall mean that these data satisfy \eqref{eq:F1} and the $3$-form version of the anomaly cancellation condition \eqref{eq:anomaly.def}, which we rewrite here: 
\begin{equation}
\label{eq:anomaly2}
    2\textrm{d}^c\omega
    =\alpha'(\text{CS}(A)-\text{CS}(\nabla))+\textrm{d} B\:.
\end{equation}
The \emph{D-terms} will be \eqref{eq:D1}--\eqref{eq:D2} as before.

\subsection{Review} In \cites{Anderson2014,Ossa2014a} it was shown that the infinitesimal moduli of the heterotic $\SU(3)$ system are counted by a cohomology similar to the one studied here. In \cite{GarciaFernandez2017}, the moduli were studied with a viewpoint motivated by generalised geometry and the moduli problem was shown to be elliptic. However, as pointed out in Remark \ref{rem:HYM1}, in each of these cases an additional instanton connection $\tilde\nabla$ on the tangent bundle of the manifold is introduced, giving rise to "spurious" degrees of freedom that are not justified from a physics perspective.

These extra degrees of freedom were removed from the analysis of \cite{McOrist2022}, where it was found that the infinitesimal moduli which parameterise deformations of the F-terms are counted by the kernel of an operator $\bar D$, similar to the operator studied in this paper.  
It was then shown in \cite{McOrist2022} that, in order that the deformations additionally solve the D-terms \eqref{eq:D1}-\eqref{eq:D2}, they should also be in the kernel of the adjoint operator $\bar D^*$, with respect to a suitable non-degenerate pairing on $Q$. It follows from Hodge theory that the infinitesimal deformations are parametrised by $H^{0,1}_{\bar D}(Q)$.
However, the analysis of \cite{McOrist2022} was only done to a physics rigour, up to  ${\cO}(\alpha'^2)$ corrections. We shall take some steps towards mathematically remedying that situation in this appendix.

%%%%%%%%%%%
\subsection{Infinitesimal deformations and cohomology}

To begin, note that simultaneous infinitesimal deformations of the holomorphic bundle and complex structure are provided by $(\gamma,W)$ in the kernel of $\bar D_1$, see \eqref{eq:barD1}, which defines the Atiyah algebroid \cite{Atiyah1957}.  In other words,
\begin{equation}\label{eq:gamma.W}
    \delbar_E\gamma+\sF W=0\quad\text{and}\quad \delbar W=0,
\end{equation}
where $\sF$ is given by the curvature $F$ of $A$ as in Definition \ref{def:F}. Here $W\in\Omega^{0,1}(T^{1,0}X)$ denotes the deformation of the complex structure, or Beltrami differential, while $\gamma\in\Omega^{0,1}(\End(E))$ is the deformation of the holomorphic connection. The deformation of the anomaly constraint \eqref{eq:anomaly2} can be written as
\begin{equation}
\label{eq:DefAnomaly}
    \delbar \kappa+\alpha'\sF\gamma+\sT W-\alpha'\tr(R\wedge \beta)=0\:.
\end{equation}
where $R$ is the curvature of $\nabla$ and $\beta\in\Omega^{(0,1)}(\End(TX))$ is the deformation of $\nabla$. Here $\kappa\in\Omega^{1,1}(X)$ denotes the deformation of the hermitian form $\omega$, complexified by the deformation of the $B$-field, and twisted at $\mathcal{O}(\alpha')$ by the Green--Schwarz mechanism to get a well-defined 2-form.

Physics dictates that $\beta$ must be a function of the other deformations involved. Furthermore, it must preserve holomorphicity of $\End(TX)$, meaning it satisfies the equation given locally by
\begin{equation}\label{eq:holo.app}
    W^a\wedge R_a{}^b{}_c =\delbar \beta^b{}_c\:,
\end{equation}
where we have suppressed antiholomorphic indices.  Using Corollary \ref{cor:commute}, we can rewrite \eqref{eq:holo.app} as:
\begin{equation}
\delbar\left(\nabla^+_c W^b\right)+\delbar\beta^b{}_c=0\:.
\end{equation}
 It follows that
\begin{equation}\label{eq:lambda}
    \beta^b{}_c=-\nabla^+_c W^b+\lambda^b{}_c\:,
\end{equation}
where $\lambda$ is a $\delbar$-closed function of the other deformations. 

We now demonstrate the following.

\begin{lemma}\label{lem:lambda}
For $\lambda$ as given by \eqref{eq:lambda}, we have that $\tr(R\wedge\lambda)$ is $\delbar$-exact.
\end{lemma}

\begin{proof}
For $\tr(R\wedge\lambda)$ to be non-trivial in $\delbar$-cohomology, there must exist a $\delbar$-harmonic $(1,2)$-form $\chi$ so that the inner-product
\begin{equation}\label{eq:chi}
\left(\tr(R\wedge\lambda),\chi\right)=\int_X\tr(R\wedge\lambda)\wedge*\bar\chi
\end{equation}
is non-zero.  If $\chi$ is such a $\delbar$-harmonic form,  note that $*\bar\chi$ is also $\delbar$-harmonic, and may be written as 
\begin{equation}\label{eq:starchi}
*\bar\chi= V\lrcorner\Omega %\tfrac12 V^a\wedge\Omega_{abc}dz^{bc}\:,
\end{equation}
for some $\delbar$-harmonic $V\in\Omega^{0,1}(T^{1,0}X)$.  Note that  the contraction in \eqref{eq:starchi} involves a wedge product on the form part of $V$.  %corresponding to some complex structure deformation. 
Some elementary manipulations using \eqref{eq:starchi} give that the inner product in \eqref{eq:chi} can be written for
\begin{equation}\label{eq:chi.2}
\left(\tr(R\wedge\lambda),\chi\right)=\int_X %\tilde\mu^a\wedge\tr(R_a\wedge\gamma)
\tr (V\lrcorner R\wedge\lambda)\wedge\Omega=\int_X\tr(\delbar\mu\wedge\lambda)\wedge\Omega\:
\end{equation}
for some $\mu$.  Clearly, the right-hand side  of \eqref{eq:chi.2} is zero, which shows that $\tr(R\wedge\lambda)$ is $\delbar$-exact.
\end{proof}

By Lemma \ref{lem:lambda} if we substitute the right-hand side of \eqref{eq:lambda} into \eqref{eq:DefAnomaly}, this term may be absorbed into $\delbar \kappa$ by an ${\cO}(\alpha')$ field-redefinition of $\kappa$, which we still call $\kappa$. Equation \ref{eq:DefAnomaly} then becomes
\begin{equation}
\label{eq:DefAnomaly2}
    \delbar \kappa+\alpha'\sF\gamma+\sT W+\alpha'\cR\nabla^+W%{R^c}_b\nabla^+_c\mu^d
    =0\:.
\end{equation}
This is precisely the $(T^{1,0}X)^*$-row of the equation $\bar D s=0$ for our operator $\bar D$ in \eqref{eq: barD}. 

We then deduce the following from \eqref{eq: barD}, \eqref{eq:gamma.W} and \eqref{eq:DefAnomaly2}.

\begin{prop}
     Infinitesimal deformations of the F-terms are parameterised by  $s\in\Gamma(Q^{0,1})$ satisfying $$\bar{D}s=0.$$
\end{prop}

As shown in \cite{McOrist2022}, in order to also preserve the D-terms, we must further impose
\begin{equation}
\label{eq:alphaprimesquared}
    \bar D^*s=t=\mathcal{O}(\alpha'^2)\:
\end{equation}
on elements $s$ of the kernel of $\bar{D}$, for some $t$ 
 of $\mathcal{O}(\alpha'^2)$ which is linear in $s$ but yet to be fully determined. It should however be noted that the Hull--Strominger system itself, and in particular the D-terms, receive corrections at $\mathcal{O}(\alpha'^2)$: that is, such terms are neglected when the system is derived from heterotic String Theory. Therefore, as far as physics is concerned, we are warranted in studying $H^{0,1}_{\bar D}(Q)$ as the cohomology parametrising infinitesimal moduli.  In general, we are only currently guaranteed that some subspace of $H^{0,1}_{\bar{D}}(Q)$ gives the infinitesimal deformations of the heterotic $\SU(3)$ system.

\begin{remark} 
In the case where $H^0_{\bar D}(Q)=0$, such as when the individual bundles have no holomorphic sections, the $\mathcal{O}(\alpha'^2)$ term $t$ in \eqref{eq:alphaprimesquared} must be $\bar D^*$-exact.  If $X$ is compact and $\alpha'$ is sufficiently small, Proposition \ref{prop:elliptic} allows us to deduce a version of Hodge decomposition for $\bar{D}$, which implies that $t=\bar{D}^*\bar{D}u$ for some $u$.   Therefore, we may replace $s$ by $s-\bar{D}u$, which means that the higher order corrections to $s$ arising from \eqref{eq:alphaprimesquared} can only serve to change the explicit representative of the $\bar D$-cohomology class $[s]\in H^{0,1}_{\bar{D}}(Q)$.  We deduce that $H^{0,1}_{\bar{D}}(Q)$ indeed parameterises the infinitesimal moduli for the heterotic $\SU(3)$ system in this case.  We expect this to be the generic behaviour, but leave a full investigation to future work.  
\end{remark}

\subsection*{Declarations}
\quad

\medskip
\noindent\textbf{Funding:}
HdL was supported by the  São Paulo Research Foundation (Fapesp) {[2020/15525-5, 2021/11442-0, 2021/11467-3]}.
JDL was partially supported by the Simons Collaboration on Special Holonomy
in Geometry, Analysis, and Physics (\#724071 Jason Lotay).
HSE was supported by the Fapesp    \mbox{[2021/04065-6]} \emph{BRIDGES collaboration} and a Brazilian National Council for Scientific and Technological Development (CNPq) Productivity Grant level 1D \mbox{[311128/2020-3]}.

JDL and HSE benefited from a Newton Mobility  bilateral collaboration (2019-2023), granted by the UK Royal Society [NMG$\backslash$R1$\backslash$191068].  During the final stages of this project, JDL was a Simons Visiting Professor for the Fall 2024 semester at SLMath (formerly MSRI) in Berkeley, California, supported by the National Science Foundation (Grant No.~DMS-1928930). 

\medskip
\noindent\textbf{Data availability:} The authors declare that all data supporting the findings of this study are available within the paper and cited references.

\medskip
\noindent\textbf{Competing interests:} The authors have no competing interests to declare that are relevant to the content of this article.

\end{document}